\documentclass[reqno]{amsart}

\usepackage{color}
\usepackage[foot]{amsaddr}
\usepackage{booktabs}
\usepackage{bigstrut}
\usepackage{amsmath,amsthm,bm}
\usepackage{amsfonts}
\usepackage{algorithmic}
\usepackage{algorithm}
\usepackage{multirow}
\usepackage{overpic}
  \usepackage{paralist}
  \usepackage{graphics} 
  \usepackage{epsfig} 
\usepackage{graphicx}
\usepackage{float}
\usepackage{subfigure}
\usepackage{epstopdf} 

\usepackage{tikz}
\usetikzlibrary{calc}

\usepackage{url}

\newtheorem{theorem}{Theorem}[section]

\newtheorem{lemma}[theorem]{Lemma}
\newtheorem{proposition}[theorem]{Proposition}

\newtheorem{definition}[theorem]{Definition}
\theoremstyle{definition}

\graphicspath{{images/}}

\begin{document}

\title[Hypergraph $p$-Laplacian equations]
{Hypergraph $p$-Laplacian equations for data interpolation and semi-supervised learning}

\author{Kehan Shi $^{\ast}$, Martin Burger $^{\dagger,\ddagger}$}
\address{$^{\ast}$Department of Mathematics, China Jiliang University,
Hangzhou 310018, China}
\address{$^{\dagger}$Computational Imaging Group and Helmholtz Imaging, Deutsches Elektronen-Synchrotron DESY, 22607 Hamburg, Germany}
\address{$^{\ddagger}$Fachbereich Mathematik, Universit\"{a}t Hamburg, 20146 Hamburg, Germany}

\email{kshi@cjlu.edu.cn}

\subjclass{35R02, 65D05}
 \keywords{Hypergraph, $p$-Laplacian, data interpolation, semi-supervised learning}


 \begin{abstract}
 Hypergraph learning with $p$-Laplacian regularization has attracted a lot of attention due to its flexibility in modeling higher-order relationships in data.
 This paper focuses on its fast numerical implementation, which is challenging due to the non-differentiability of the objective function and the non-uniqueness of the minimizer.
We derive a hypergraph $p$-Laplacian equation from the subdifferential of the $p$-Laplacian regularization.
 A simplified equation that is mathematically well-posed and computationally efficient is proposed as an alternative.
Numerical experiments verify that the simplified $p$-Laplacian equation suppresses spiky solutions in data interpolation and improves classification accuracy in semi-supervised learning.
The remarkably low computational cost enables further applications.
\end{abstract}

\maketitle

\section{Introduction}
Over the past two decades, hypergraphs have become a valuable tool for data processing.
It is defined as the generalization of a graph in which a hyperedge can connect more than two vertices.
This allows hypergraphs to model higher-order relations involving multiple vertices in data, with applications in areas including image processing \cite{zhang2020hypergraph,shi2024continuum}, bioinformatics \cite{klamt2009hypergraphs,patro2013predicting},  social networks \cite{antelmi2021social,fazeny2023hypergraph}, etc.

In this paper, we focus on semi-supervised learning on undirected hypergraphs.
Let $H=(V, E, W)$ be a hypergraph, where $V=\{x_i\}_{i=1}^n\subset \mathbb{R}^d$ denotes the vertex set, $E=\{e_k\}_{k=1}^m$ is the hyperedge set, and $W=\{w_k\}_{k=1}^m$ assigns positive weights for hyperedges.
We are given a subset of labeled vertices $\{x_i\}=:L\subset V$ and the associated labels $\{y_i\}\subset \mathbb{R}$.
The goal is to assign labels for the remaining vertices $V\backslash L$ based on the training data $\{(x_i, y_i), x_i\in L, y_i\in \mathbb{R}\}$.

We consider the standard approach that minimizes
the constraint functional
\begin{equation*}
  F^{con}(u)=
  \begin{cases}
    F(u), \quad &\mbox{if } u(x_i)=y_i,~ x_i\in L,\\
  +\infty,\quad &\mbox{otherwise.}
  \end{cases}
\end{equation*}
Here $F(u)$ is the regularization of $u: V\rightarrow\mathbb{R}$ that enforces the smoothness of $u$.
It implicitly assumes that vertices within the same hyperedge tend to have the same label.
The constraint in $F^{con}$ ensures the minimizer of it to satisfy the given training data.
One of the fundamental algorithms for hypergraph learning is the $p$-Laplacian regularization \cite{zhou2006learning}
\begin{equation}\label{eq:1.0}
  F_{CE}(u)=\sum_{k=1}^m w_k \sum_{x_i,x_j\in e_k}|u(x_i)-u(x_j)|^p,
\end{equation}
where $p>1$.
Notice that the hypergraph $H$ can be approximated by a weighted graph $G=(V,E_G,W_G)$ with clique expansion \cite{agarwal2006higher}.
More precisely, for any $x_i,x_j\in e_k\subset E$, there exists an edge $e_{i,j}\in E_G$. The associated weight $w_{i,j}=\frac{w_k}{C_{|e_k|}^2}$, where $|e_k|$ denotes the cardinality of hyperedge $e_k$
and $C^2_{|e_k|}=|e_k|(|e_k|-1)/2$.
Then
functional $F_{CE}$ is equivalent to applying the graph $p$-Laplacian \cite{el2016asymptotic}
\begin{equation}\label{eq:1.2}
  F_{G}(u)=\sum_{i,j=1}^n w_{i,j} |u(x_i)-u(x_j)|^p,
\end{equation}
to the weighted graph $G$.
It was shown that the approximation approach can not fully utilize the hypergraph structure \cite{agarwal2006higher}.
Later in \cite{hein2013total},
the authors proposed to overcome this limitation with a new hypergraph $p$-Laplacian regularization
\begin{equation}\label{eq:1.1}
  F_H(u)=\sum_{k=1}^m w_k \max_{x_i,x_j\in e_k}|u(x_i)-u(x_j)|^p,
\end{equation}
which is deduced from the Lov\'{a}sz extension of the hypergraph cut.

$F_H$ is more mathematically appealing than $F_{CE}$ due to its convexity but non-differentiability.
In \cite{ikeda2023nonlinear}, the authors defined the hypergraph $p$-Laplacian operator, which is multivalued, as the subdifferential $\partial F_H$.
Properties of solutions to nonlinear evolution equations governed by $\partial F_H$
(i.e., $\frac{du(t)}{dt}+\partial F_Hu(t)\ni 0$ and its variants)
 were studied \cite{ikeda2023nonlinear,fukao2022heat}.
The variational consistency between $F_H^{con}$ and the continuum $p$-Laplacian
\begin{equation*}
  \mathcal{F}^{con}(u)=
  \begin{cases}
    \int_{\Omega}|\nabla u|^pdx, \quad &\mbox{if } u\in W^{1,p}(\Omega) \mbox{ and } u(x_i)=y_i,~ x_i\in L,\\
  +\infty,\quad &\mbox{otherwise,}
  \end{cases}
\end{equation*}
for $p>d$ was established in our previous paper \cite{shi2024hypergraph} in the setting when the number of vertices $n$ goes to
infinity while the number of labeled vertices remains fixed.
To avoid the complicated structure of the hypergraph,
we considered the $\varepsilon_n$-ball hypergraph constructed from point cloud data by the distance-based method.
This is analogous to the graph case \cite{slepcev2019analysis}, in which the variational consistency between $F_G^{con}$ and $\mathcal{F}^{con}$ was obtained for a random geometric graph with connection radius $\varepsilon_n$.
Compared to $F_G^{con}$, the result for $F_H^{con}$ relies on a weaker assumption on the upper bound of $\varepsilon_n$.
Numerical experiments in \cite{shi2024hypergraph} show that $F_H^{con}$ suppresses spiky solutions in data interpolation better than $F_G^{con}$.

On the other hand, the non-differentiability of $F_H^{con}$ and the non-uniqueness of its minimizers cause some challenges in the numerical aspect.
Unlike the graph functional $F_{G}^{con}$, there exists no straightforward and efficient algorithm for minimizing $F_{H}^{con}$, even in the case $p=2$.
The primal-dual hybrid gradient (PDHG) algorithm \cite{chambolle2011first} was first considered in \cite{hein2013total} for $p=1,2$.
A new algorithm \cite{shi2024hypergraph} that works any $p\geq 1$ was proposed based on the stochastic PDHG algorithm \cite{chambolle2018stochastic}.
To avoid the non-uniqueness of minimizers for $F_H^{con}$, the $\ell^2$-norm constraint was used in \cite{hein2013total}.
While in \cite{zhang2017re}, the authors proposed to minimize $F_H^{con}$ by the subgradient descent method \cite{shor2012minimization} and utilized the confidence interval to ensure the uniqueness.
Nevertheless, their high computational cost cannot be neglected for large-scale datasets and hinders further applications of the hypergraph $p$-Laplacian.

The purpose of this paper is to provide an alternative to $F_H^{con}$ that can be uniquely and efficiently solved.
We begin by addressing the non-uniqueness issue of the minimizer for $F_H^{con}$ and obtain a single-valued $p$-Laplacian operator from the subdifferential of $F_H^{con}$.
The operator involves unknown parameters depending on the structure of the hypergraph, which prevents us from solving the associated $p$-Laplacian equations for semi-supervised learning.
A simplified equation that disregards these parameters is then proposed as an approximation.
It is mathematically well-posed: It admits a unique solution and satisfies the comparison principle.
There exist hypergraphs on which the solution of the simplified equation coincides with a minimizer of $F_H^{con}$.
Despite this simplification, we still refer to it as the hypergraph $p$-Laplacian equation.

Through numerical experiments on one-dimensional data interpolation, we observe that the simplified hypergraph $p$-Laplacian equation substantially inherits the characteristic of  $F_H^{con}$ that suppresses spiky solutions.
Experimental results on real-world datasets indicate that it even improves the classification accuracy for semi-supervised learning.
The most notable feature of the new equation is its low computational cost.
Compared to the aforementioned algorithms, it dramatically reduces the computation time.
Further applications of the hypergraph $p$-Laplacian for large-scale datasets become possible.

This paper is organized as follows. In section 2, we establish a hypergraph $p$-Laplacian equation from the subdifferential of $F_H^{con}$ and propose a simplified version that is computationally feasible and efficient. The properties of solutions for the equation are also discussed.
Numerical experiments are presented in section 3 to demonstrate the performance of the simplified equation for data interpolation and semi-supervised learning. We conclude this paper in section 4.

\section{Hypergraph $p$-Laplacian equations}
Let $p>1$.
Throughout this paper, we always assume that the hypergraph $H$ is connected. Namely,
for any $x_i,x_j\in V$, there exist hyperedges $e_{k_1},e_{k_1},\cdots,e_{k_l}\in E$, such that
$x_i\in e_{k_1}$, $x_j\in e_{k_l}$, and $e_{k_s}\cap e_{k_{s+1}}\neq \emptyset$ for any $s=1,\cdots,l-1$.

\subsection{The property for the minimizer of $F_H^{con}$}
Notice that $F_H^{con}$ is coercive and lower semi-continuous  \cite{shi2024hypergraph}, it admits at least one minimizer.
The non-uniqueness of minimizers can be seen from the fact that the functional depends only on the maximum and minimum values on each hyperedge.
We are more concerned with vertices whose values are uniquely determined when minimizing the functional $F_H^{con}$.

\begin{definition}
  Let $u$ be a minimizer of $F_H^{con}$. We define $D(u)\subset V$ to be a subset of vertices such that for any $x_i\in D(u)$ and any perturbation of $u$ at $x_i$, i.e.,
  \begin{align*}
    v(x_j):=
    \begin{cases}
      u(x_i)+\varepsilon, \quad &\mbox{if } j=i,\\
      u(x_j),\quad &\mbox{otherwise},
    \end{cases}
  \end{align*}
  where $\varepsilon$ is a constant with a small absolute value, $F_H^{con}(v)>F_H^{con}(u)$ holds.
\end{definition}

Clearly, $L\subset D(u)$. The following lemma that characterizes the vertex in $D(u)\backslash L$ follows from  the definition directly.
\begin{lemma}\label{le:2.1:0}
  Let $u$ be a minimizer of $F_H^{con}$ and $x_i\in D(u)\backslash L$. There exist two hyperedges $e_k, e_l\in E$ such that
  \begin{equation}\label{eq:2.1:1}
    u(x_i)=\max_{x_j\in e_k}u(x_j)\quad \mbox{and}\quad u(x_i)=\min_{x_j\in e_l}u(x_j).
  \end{equation}
\end{lemma}

The maximum and minimum values of a minimizer on each hyperedge are uniquely determined by $F_H^{con}$.
\begin{lemma}\label{le:2.1:0a}
    Let $u_1$ and $u_2$ be two minimizers of $F_H^{con}$.
    For any $e_k\in E$, we have
  \begin{equation}\label{eq:2.1:3}
    \max_{x_j\in e_k}u_1(x_j)=\max_{x_j\in e_k}u_2(x_j)\quad\mbox{and}\quad
    \min_{x_j\in e_k}u_1(x_j)=\min_{x_j\in e_k}u_2(x_j).
  \end{equation}
\end{lemma}

\begin{proof}
Define $u=\frac{u_1+u_2}{2}$. For any $e_k\in E$, there exist $x_{k,i}, x_{k,j}\in V$, such that
\begin{align*}
  \max_{x_i,x_j\in e_k}|u(x_i)&-u(x_j)|^p =|u(x_{k,i})-u(x_{k,j})|^p \\
  &\leq \frac{1}{2}|u_1(x_{k,i})-u_1(x_{k,j})|^p +\frac{1}{2} |u_2(x_{k,i})-u_2(x_{k,j})|^p \\
  &\leq \frac{1}{2}\max_{x_i,x_j\in e_k}|u_1(x_i)-u_1(x_j)|^p
  +\frac{1}{2}\max_{x_i,x_j\in e_k}|u_2(x_i)-u_2(x_j)|^p.
\end{align*}
If any one of the above inequalities is strict, we have
\begin{equation*}
  F_H(u)< \frac{1}{2}\left(F_H(u_1)+F_H(u_2)\right)=F_H(u_1),
\end{equation*}
which contradicts the assumption that $u_1$ is a minimizer of $F_H^{con}$.
Consequently, by the strict convexity of $|\cdot|^p$,
\begin{equation*}
  \max_{x_i,x_j\in e_k}|u_1(x_i)-u_1(x_j)|^p=\max_{x_i,x_j\in e_k}|u_2(x_i)-u_2(x_j)|^p,
\end{equation*}
for any $e_k\in E$.
This yields
\begin{equation*}
  \max_{x_i,x_j\in e_k}\left(u_1(x_i)-u_1(x_j)\right)=\max_{x_i,x_j\in e_k}\left(u_2(x_i)-u_2(x_j)\right)
\end{equation*}
and there exists a constant $C_k$, such that
  \begin{equation*}
    \max_{x_j\in e_k}u_1(x_j)=\max_{x_j\in e_k}u_2(x_j)+C_k\quad\mbox{and}\quad
    \min_{x_j\in e_k}u_1(x_j)=\min_{x_j\in e_k}u_2(x_j)+C_k.
  \end{equation*}
  It is not difficult to see from the fact $u_1(x_i)=u_2(x_i)=y_i$ for $x_i\in L$ that $C_k=0$.
  This proves \eqref{eq:2.1:3}.
\end{proof}

\begin{proposition}\label{le:2.1:1}
  If $u_1$ and $u_2$ are two minimizers of $F_H^{con}$, then $D(u_1)=D(u_2)$ and
  \begin{equation}\label{eq:2.1:2}
    u_1(x_i)=u_2(x_i),
  \end{equation}
  for any $x_i\in D:=D(u_1)=D(u_2)$.
\end{proposition}
The proposition implies that $D$ is uniquely determined by $F_H^{con}$. The non-uniqueness of minimizers for $F_H^{con}$ comes from $D^c:=V\backslash D$.

\begin{proof}
  Let $x_i\in D(u_1)\backslash L$. By \eqref{eq:2.1:1}, we assume w.l.o.g. that
  $x_i\in e_k\cap e_l$ and
  \begin{equation}\label{eq:2.1:3a}
    u_1(x_i)=\max_{x_j\in e_k}u_1(x_j)=\min_{x_j\in e_l}u_1(x_j),
  \end{equation}
 for two hyperedges $e_k,e_l\in E$.
  By \eqref{eq:2.1:3}, to prove \eqref{eq:2.1:2}, we only need to show that
  \begin{equation*}
    u_2(x_i)=\max_{x_j\in e_k}u_2(x_j),
  \end{equation*}
  which also implies that $x_i\in D(u_2)$ and proves that $D(u_1)=D(u_2)$.

  If this is not true, i.e.,
  \begin{equation*}
    u_2(x_i)<\max_{x_j\in e_k}u_2(x_j),
  \end{equation*}
  it follows from \eqref{eq:2.1:3} and \eqref{eq:2.1:3a} that
  \begin{equation*}
    u_2(x_i)<\max_{x_j\in e_k}u_2(x_j)=\max_{x_j\in e_k}u_1(x_j)=u_1(x_i)=\min_{x_j\in e_l}u_1(x_j)=\min_{x_j\in e_l}u_2(x_j).
  \end{equation*}
  This contradicts the assumption that $x_i\in e_l$ and finishes the proof.
\end{proof}

\subsection{The subdifferential of $F_H^{con}$ and the hypergraph $p$-Laplacian equation}
Functional $F_H$ and $F_H^{con}$ are non-differentiable. We consider the subdifferential for them.

Let $\mathcal{H}$ be a Hilbert space with inner product $\langle\cdot,\cdot\rangle$ and induced norm $\|\cdot\|=\sqrt{\langle\cdot,\cdot\rangle}$.
For a proper, convex, and lower semi-continuous functional
$J: \mathcal{H}\rightarrow \mathbb{R}\cup\{+\infty\}$ with
 effective domain
\begin{equation*}
  \mbox{dom}(J)=\{u\in\mathcal{H}: J(u)<\infty\},
\end{equation*}
the subdifferential of $J$ at $u\in\mbox{dom}(J)$ is defined as
\begin{equation*}
  \partial J(u)=\left\{q\in\mathcal{H}: \langle q, v-u\rangle\leq J(v)-J(u), \forall v\in \mathcal{H}\right\}.
\end{equation*}
An element of $\partial J(u)$ is called a subgradient of $J$ at $u$.
The subgradient coincides with the usual gradient if $J$ is differentiable.
We shall use the following proposition of the subdifferential
\begin{equation}\label{eq:2.2:1}
  \mbox{$u$ is a minimizer of $J$ } \Longleftrightarrow  0\in \partial J(u),
\end{equation}
whose proof is trivial.

The subdifferential of $F_H$ has been obtained in \cite{ikeda2023nonlinear}.
More precisely,
\begin{align}\label{eq:2.2:3}
  \partial F_H(u)=p\left\{\sum_{k=1}^m w_k\max_{x_i,x_j\in e_k}|u(x_i)-u(x_j)|^{p-1}b_{k},~~ b_{k}\in\arg\max_{b\in B_{k}}\langle b,u\rangle\right\},
\end{align}
where
\begin{equation*}
  B_{k}=\mbox{conv}\{\mathbb{I}_{x_i}-\mathbb{I}_{x_j}: x_i,x_j\in e_k\},
\end{equation*}
$\mathbb{I}_{x_i}\in\mathbb{R}^n$ is an indicator function
\begin{align*}
  \mathbb{I}_{x_i}(x_j)=
  \begin{cases}
    1,\quad \mbox{if }x_j=x_i,\\
    0,\quad \mbox{otherwise},
  \end{cases}
\end{align*}
and $\mbox{conv}\{S\}$ denotes the convex hull of $S$ in $\mathbb{R}^n$.

The subdifferential for the constraint functional $\partial F_H^{con}$ is a corollary of the definition and \eqref{eq:2.2:3}.
For any
\begin{equation*}
  u\in \mbox{dom}(F_H^{con})=\{v\in\mathbb{R}^n: v(x_i)=y_i,  x_i\in L\},
\end{equation*}
we have
\begin{equation}\label{eq:2.2:4}
  \partial F_H^{con}(u)=\left\{q\in\mathbb{R}^n: q(x_i)=h(x_i),  x_i\in V\backslash L \mbox{ for any }  h\in \partial F_H(u)\right\}.
\end{equation}
Namely, a subgradient of $F_H^{con}(u)$ comes from a subgradient of $F_H(u)$ by taking arbitrary values at labeled vertices.

By combining the above results \eqref{eq:2.2:1}--\eqref{eq:2.2:4}, we deduce an equivalent form for the minimizer of $F_H^{con}$.
More precisely,
if $u$ is a minimizer of $F_H^{con}$,
there exist vectors $\beta_{k}\in\arg\max_{b\in B_{k}}\langle b,u\rangle$ such that
\begin{align}\label{eq:2.2:2}
  \begin{cases}
    \left(\sum_{k=1}^m w_k\max_{x_i,x_j\in e_k}|u(x_i)-u(x_j)|^{p-1}\beta_{k}\right)(x_i)=0,\quad x_i\in V\backslash L,\\
    u(x_i)=y_i,\quad  x_i\in L.
  \end{cases}
\end{align}
Conversely, if a function $u$ satisfies equation \eqref{eq:2.2:2}, where $\beta_{k}\in\arg\max_{b\in B_{k}}\langle b,u\rangle$, then it is a minimizer of $F_H^{con}$.

In the rest of this subsection, we propose a new hypergraph equation based on equation \eqref{eq:2.2:2}.
The basic idea is to consider $|\beta_{k}(x_i)|$ as a diffusion coefficient that represents the contribution of hyperedge $e_k$ to vertex $x_i$.
The proposed equation reads
\begin{align}\label{E}
  \begin{cases}
    \sum\limits_{k=1}^mw_k\alpha_{k}(x_i)
    \left|\max\limits_{x_j\in e_k\cap D}u(x_j)+\min\limits_{x_j\in e_k\cap D}u(x_j)-2u(x_i)\right|^{p-2} \\
    \qquad\qquad\quad\left(\max\limits_{x_j\in e_k\cap D}u(x_j)+\min\limits_{x_j\in e_k\cap D}u(x_j)-2u(x_i)\right)=0,\quad x_i\in D\backslash L,\\
    u(x_i)=y_i,\quad  x_i\in L,
  \end{cases}
\end{align}
where
\begin{align*}
  \alpha_{k}(x_i)=
  \begin{cases}
    |\beta_ {k}(x_i)|,\quad &\mbox{if } x_i\in e_k, \\
    0,\quad & \mbox{if } x_i\notin e_k,
  \end{cases}
\end{align*}
for $e_k\in E$ and $x_i\in D$.
Here we restrict the equation  on the subhypergraph
$\tilde{H}=(D, \tilde{E},W)$, where $\tilde{E}=\{e_k\cap D\}_{k=1}^m$, to avoid the non-uniqueness.
The notation $|0|^{p-2}0=0$ is used for the case $1<p<2$.

Owing to the following theorem, we call equation \eqref{E} the hypergraph $p$-Laplacian equation.
\begin{theorem}\label{th:existence}
  Let $u$ be a minimizer of $F_H^{con}$. Then $u|_D$
  is a solution of equation \eqref{E}.
\end{theorem}

\begin{proof}
  We redefine $u$ in $D^c$ such that for any $e_k\in E$,
  \begin{equation*}
    \min_{x_i\in e_k}u(x_j)< u(x_i)<\max_{x_i\in e_k}u(x_j),\quad x_i\in e_k\cap D^c.
  \end{equation*}
  By the assumption, $u$ and $\beta_{k}$ satisfy equation \eqref{eq:2.2:2}.
  Notice that for any $e_k\in E$ and $x_i\in D^c$,
  \begin{equation*}
    \beta_{k}(x_i)=0.
  \end{equation*}
  Namely, equation \eqref{eq:2.2:2} is trivial for $x_i\in D^c$.

  Let $x_i\in D\backslash L$ and $e_k\in E$.
  If $\beta_{k}(x_i)>0$, we have
  \begin{equation*}
    x_i\in \arg\max_{x_j\in e_k}u(x_j),
  \end{equation*}
  and
  \begin{align*}
    &-\max_{x_i,x_j\in e_k}|u(x_i)-u(x_j)|^{p-1}\beta_{k}(x_i)\\
    &=\alpha_{k}(x_i)
    \left|\max\limits_{x_j\in e_k}u(x_j)+\min\limits_{x_j\in e_k}u(x_j)-2u(x_i)\right|^{p-2}
    \left(\max\limits_{x_j\in e_k}u(x_j)+\min\limits_{x_j\in e_k}u(x_j)-2u(x_i)\right).
  \end{align*}
  The same conclusion can be obtained for the cases $\beta_{k}(x_i)<0$ and $\beta_{k}(x_i)=0$.
  Notice that
  \begin{equation*}
    \max_{x_j\in e_k}u(x_j)=\max_{x_j\in e_k\cap D}u(x_j),\quad
    \min_{x_j\in e_k}u(x_j)=\min_{x_j\in e_k\cap D}u(x_j).
  \end{equation*}
  This means that $u(x_i)$ satisfies equation \eqref{E} and finishes the proof.
\end{proof}

Conversely, if $v$ is a solution of \eqref{E}, it is not difficult to verify by reversing the proof of Theorem \ref{th:existence} that $v$ is also a minimizer of $F_{\tilde{H}}^{con}$.
Then a unique minimizer of $F_H^{con}$ can be determined, e.g.,
\begin{align*}
  u(x_i)=
  \begin{cases}
    v(x_i),\quad \mbox{if } x_i\in D,\\
    \frac{1}{2}\left(\max\limits_{1\leq j\leq l}\min\limits_{x_i\in e_{k_j}\cap D}v(x_i)
    +\min\limits_{1\leq j\leq l}\max\limits_{x_i\in e_{k_j}\cap D}v(x_i)
    \right),\quad \mbox{if } x_i\in D^c.
  \end{cases}
\end{align*}
Here we use the notation for $x_i\in D^c$ that $|x_i|=l$ (i.e., the degree of $x_i$ is $l$) and $x_i\in e_{k_1},\cdots,e_{k_l}$.

Although a minimizer of $F_H^{con}$ can be uniquely determined through equation \eqref{E}. The equation itself is not solvable numerically. Indeed, both the domain $D$ and the diffusion coefficient $\alpha_k(x_i)$ depend on the structure of the hypergraph and the training set and thus have no general expression.

\subsection{A simplified hypergraph $p$-Laplacian equation}

The purpose of this subsection is to present a simplified version of equation \eqref{E} that does not involve $D$ and $\alpha_k(x_i)$.
To this end, we consider the homogeneous coefficient $\alpha_k(x_i)\equiv 1$ and the whole domain $V$.
The new equation is as follows
\begin{align}\label{AE}
  \begin{cases}
    L^p_H u:=\sum\limits_{k=1}^m w_k\chi_k(x_i)
    \left|\max\limits_{x_j\in e_k}u(x_j)+\min\limits_{x_j\in e_k}u(x_j)-2u(x_i)\right|^{p-2}\\
    \qquad\qquad\qquad\quad
    \left(\max\limits_{x_j\in e_k}u(x_j)+\min\limits_{x_j\in e_k}u(x_j)-2u(x_i)\right)=0,\quad x_i\in V\backslash L,\\
    u(x_i)=y_i,\quad x_i\in L,
  \end{cases}
\end{align}
where
\begin{align*}
  \chi_{k}(x_i)=
  \begin{cases}
    1,\quad &\mbox{if } x_i\in e_k, \\
    0,\quad & \mbox{if } x_i\notin e_k,
  \end{cases}
\end{align*}
for $e_k\in E$ and $x_i\in V$.

 In general, a solution of equation \eqref{AE} is no longer a solution of equation \eqref{E} (when restricting to $D$) and is not a minimizer of $F_{H}^{con}$.
 Figure \ref{fig:2.2:2} shows an example of a hypergraph and a function $u$ on it that minimizes $F_H^{con}$.
 Clearly, $u$ is not a solution of equation \eqref{AE} since in this case $\alpha_k(x_i)\neq \chi_k(x_i)$.
Nevertheless, there exist specific instances where a solution of equation \eqref{AE} and a minimizer of $F_{H}^{con}$ coincide, as demonstrated in Figure \ref{fig:2.4:1}.
For this reason, we still refer to equation \eqref{AE} as the hypergraph $p$-Laplacian equation.
The theoretical study of the connection between the discrete equation \eqref{AE} and the classical $p$-Laplacian equation will be part of our future work.

\begin{figure}[H]
  \centering
\begin{tikzpicture}
  \node (x1) at (-0.5, 2) {};
  \node (x2) at (1, 2) {};
  \node (x3) at (2, 2.5) {};
  \node (x4) at (3, 2.) {};
  \node (x5) at (4, 3.5) {};
  \node (x6) at (4., 0.5) {};

  \begin{scope}[fill opacity = 0.8]
      \filldraw [fill = yellow!70] ($(x1) + (-0.5, 0)$)
      to [out = 90, in = 180] ($(x2) + (0, 0.5)$)
      to [out = 0,in = 90] ($(x2) + (0.5, 0)$)
      to [out = 270, in = 0] ($(x1) + (0.5, -0.8)$)
      to [out = 180, in = 270] ($(x1) + (-0.5, 0)$);
      \filldraw [fill = red!70] ($(x2) + (-0.5, 0)$)
      to [out = 90, in = 180] ($(x3) + (0, 0.5)$)
      to [out = 0, in = 90] ($(x4) + (0.5, 0)$)
      to [out = 270, in = 0] ($(x4) + (0, -0.5)$)
      to [out = 180, in = 270] ($(x2) + (-0.5, 0)$);
      \filldraw [fill = green!70] ($(x4) + (-0.5, 0)$)
      to [out = 90, in = 180] ($(x5) + (0, 0.5)$)
      to [out = 0,in = 90] ($(x5) + (0.5, -0.5)$)
      to [out = 270, in = 0] ($(x4) + (0.0, -0.5)$)
      to [out = 180, in = 270] ($(x4) + (-0.5, 0)$);
      \filldraw [fill = blue!70] ($(x4) + (-0.5, 0)$)
      to [out = 90, in = 90] ($(x6) + (0.5, 0)$)
      to [out = 270,in = 0] ($(x6) + (0., -0.5)$)
      to [out = 180, in = 270] ($(x4) + (-0.5, 0)$);
  \end{scope}

  \foreach \i in {1, 2, ..., 6}
  {
      \fill (x\i) circle (0.1);
  }

  \fill (x1) circle (0.1) node [below] {$x_1=4$};
  \fill (x2) circle (0.1) node [below] {$x_2$};
  \fill (x3) circle (0.1) node [above] {$x_3=0$};
  \fill (x4) circle (0.1) node [below] {$x_4$};
  \fill (x5) circle (0.1) node [below] {$x_5=3$};
  \fill (x6) circle (0.1) node [below] {$x_6=3$};

  \begin{scope}[every node/.style = {fill, shape = circle, node distance = 25pt}]
      \node (e1) [color = yellow!56, label = right:$e_1$] at (-3, 3) {};
      \node (e2) [below of = e1, color = red!56, label = right:$e_2$] {};
      \node (e3) [below of = e2, color = green!56, label = right:$e_3$] {};
      \node (e3) [below of = e3, color = blue!56, label = right:$e_4$] {};
  \end{scope}
\end{tikzpicture}
\caption{A hypergraph with 6 vertices and 4 hyperedges. Let $p=2$, $w_k=1$, $k=1,\cdots,4$, and $x_1, x_3, x_5, x_6\in L$ be labeled vertices. Then $u=(4,\frac{5}{2},0,\frac{5}{2},3,3)^T$ is a minimizer of $F_H^{con}$ and $\beta_2(x_2)=\frac{3}{5}$, $\beta_2(x_4)=\frac{2}{5}$.}
\label{fig:2.2:2}
\end{figure}

\begin{figure}[H]
  \centering
\begin{tikzpicture}
  \node (x1) at (-0.5, 2) {};
  \node (x2) at (1, 2) {};
  \node (x3) at (2, 2.5) {};
  \node (x4) at (3, 2.3) {};
  \node (x5) at (3, 1.5) {};
  \node (x6) at (4., 3) {};
  \node (x7) at (4., 1) {};

  \begin{scope}[fill opacity = 0.8]
      \filldraw [fill = yellow!70] ($(x1) + (-0.5, 0)$)
      to [out = 90, in = 180] ($(x2) + (0, 0.5)$)
      to [out = 0,in = 90] ($(x2) + (0.5, 0)$)
      to [out = 270, in = 0] ($(x1) + (0.5, -0.8)$)
      to [out = 180, in = 270] ($(x1) + (-0.5, 0)$);
      \filldraw [fill = green!70] ($(x2) + (-0.5, 0)$)
      to [out = 90, in = 180] ($(x3) + (0, 0.5)$)
      to [out = 0, in = 90] ($(x4) + (0.5, 0)$)
      to [out = 270, in = 0] ($(x5) + (0, -0.5)$)
      to [out = 180, in = 270] ($(x2) + (-0.5, 0)$);
      \filldraw [fill = red!70] ($(x4) + (-0.5, 0)$)
      to [out = 90, in = 180] ($(x6) + (0, 1)$)
      to [out = 0, in = 90] ($(x7) + (0.5, -0.5)$)
      to [out = 270, in = 0] ($(x5) + (0, -1.2)$)
      to [out = 180, in = 270] ($(x4) + (-0.5, 0)$);
  \end{scope}

  \foreach \i in {1, 2, ..., 7}
  {
      \fill (x\i) circle (0.1);
  }

  \fill (x1) circle (0.1) node [below] {$x_1=0$};
  \fill (x2) circle (0.1) node [below right] {$x_2$};
  \fill (x3) circle (0.1) node [left] {$x_3$};
  \fill (x4) circle (0.1) node [below right] {$x_4$};
  \fill (x5) circle (0.1) node [below left] {$x_5$};
  \fill (x6) circle (0.1) node [below right] {$x_6$};
  \fill (x7) circle (0.1) node [below] {$x_7=3$};

  \begin{scope}[every node/.style = {fill, shape = circle, node distance = 30pt}]
      \node (e1) [color = yellow!56, label = right:$e_1$] at (-3, 3) {};
      \node (e2) [below of = e1, color = green!56, label = right:$e_2$] {};
      \node (e3) [below of = e2, color = red!56, label = right:$e_3$] {};
  \end{scope}
\end{tikzpicture}
\caption{A hypergraph with 7 vertices and 3 hyperedges. Let $p=2$, $w_k=1$, $k=1,2,3$, and $x_1, x_7\in L$ be labeled vertices. Then $u=(0,1,\frac{3}{2},2,2,\frac{5}{2},3)^T$ is both a solution of equation \eqref{AE} and a minimizer of $F_H^{con}$.}
\label{fig:2.4:1}
\end{figure}

The comparison principle and the unique solvability of equation \eqref{AE} are stated as follows.

\begin{theorem}[Comparison principle]\label{th:2.3:2}
  If $u_1, u_2: V\rightarrow \mathbb{R}$ are two functions that satisfy
\begin{equation*}
  L^p_{H}u(x_j)=0,\quad \mbox{for } \forall x_j\in V\backslash L,
\end{equation*}
and
\begin{equation*}
  u_1(x_j)\leq u_2(x_j),\quad \mbox{for } \forall x_j\in L,
\end{equation*}
then
\begin{equation*}
  u_1(x_j)\leq u_2(x_j),\quad \mbox{for } \forall x_j\in V.
\end{equation*}
\end{theorem}
\begin{proof}
  Assume to the contrary. We claim that there exist a hyperedge $e_{k_1}$ and vertices $x_i, x_j\in e_{k_1}$ such that
  \begin{equation*}
    u_1(x_i)-u_2(x_i)=\max_{x\in V}\left(u_1(x)-u_2(x)\right)=:c>0,
  \end{equation*}
  and
  \begin{equation}\label{eq:2.3:2}
    u_1(x_j)- u_2(x_j)< c.
  \end{equation}
  Otherwise, by the connectivity assumption of the hypergraph $H$, $u_1-u_2=c$ on every hyperedge,
  which is a contradiction.

  Assume w.l.o.g. that $|x_i|=l$ and $x_i\in e_{k_1},\cdots, e_{k_l}$.
  Then we have
  \begin{equation*}
    u_1(x_i)-u_2(x_i)\geq u_1(x)-u_2(x), \quad \forall x\in e_{k_1},\cdots, e_{k_l}.
  \end{equation*}
  Equivalently,
  \begin{equation*}
    u_2(x)-u_2(x_i)\geq u_1(x)-u_1(x_i), \quad \forall x\in e_{k_1},\cdots, e_{k_l}.
  \end{equation*}
  Taking the maximum and the minimum for the above inequality respectively and combining the results lead to
  \begin{equation}\label{eq:th:2.6:1}
    \max\limits_{x\in e}u_2(x)+\min\limits_{x\in e}u_2(x)-2u_2(x_i)
    \geq
    \max\limits_{x\in e}u_1(x)+\min\limits_{x\in e}u_1(x)-2u_1(x_i),
  \end{equation}
  for any $e=e_{k_1},\cdots, e_{k_l}$.
  It follows from \eqref{eq:2.3:2} that the above inequality is strict for hyperedge $e=e_{k_1}$.
  In fact, by
  \begin{align*}
    \min\limits_{x\in e_{k_1}}u_1(x)-\min\limits_{x\in e_{k_1}}u_2(x)
    \leq \min\limits_{x\in e_{k_1}}(u_1(x)-u_2(x))\leq u_1(x_j)-u_2(x_j)<c\\
    =u_1(x_i)-u_2(x_i),
  \end{align*}
  we have
  \begin{equation*}
    \min\limits_{x\in e_{k_1}}u_2(x)-u_2(x_i)>\min\limits_{x\in e_{k_1}}u_1(x)-u_1(x_i),
  \end{equation*}
  and consequently,
  \begin{equation*}
    \max\limits_{x\in e_{k_1}}u_2(x)+\min\limits_{x\in e_{k_1}}u_2(x)-2u_2(x_i)>\max\limits_{x\in e_{k_1}}u_1(x)+\min\limits_{x\in e_{k_1}}u_1(x)-2u_1(x_i).
  \end{equation*}
  This together with \eqref{eq:th:2.6:1} and the monotonicity of $|s|^{p-2}s$ imply that
  \begin{equation*}
    L^p_{H}u_2(x_i)> L^p_{H}u_1(x_i),
  \end{equation*}
  which is a contradiction to the assumption.
\end{proof}

\begin{theorem}
  Equation \eqref{AE} admits a unique solution $u$ that satisfies the estimate
  \begin{equation*}
    \min{y_i}\leq u(x_i)\leq\max{y_i}, \quad x_i\in  V.
  \end{equation*}
\end{theorem}

\begin{proof}
  The uniqueness of solutions is a corollary of the comparison principle. We prove the existence of a solution in the following by the Brouwer fixed-point theorem.
  It can also be proven by Perron's method.

  Let
  \begin{equation*}
    X=\left\{u\in\mathbb{R}^n: u(x_i)=y_i, x_i\in L \mbox{ and } \min{y_i}\leq u(x_i)\leq\max{y_i}, x_i\in  V\right\}
  \end{equation*}
  be a closed and convex subset of $\mathbb{R}^n$.
  For a $u\in X$,
  we consider the auxiliary equation
  \begin{align}\label{eq:th:2.5:1}
    \begin{cases}
      \sum\limits_{k=1}^m w_k\chi_k(x_i)
      \left|\max\limits_{x_j\in e_k}u(x_j)+\min\limits_{x_j\in e_k}u(x_j)-2v(x_i)\right|^{p-2}\\
      \qquad\qquad\qquad\quad
      \left(\max\limits_{x_j\in e_k}u(x_j)+\min\limits_{x_j\in e_k}u(x_j)-2v(x_i)\right)=0,\quad x_i\in V\backslash L,\\
      v(x_i)=y_i,\quad x_i\in L.
    \end{cases}
  \end{align}
  Recall that under the notation $|0|^{p-2}0=0$ for  $1<p<2$, $|s|^{p-2}s$ is continuous and monotone on $\mathbb{R}$ for $p>1$.
  Consequently, for any $x_i\in V\backslash L$, the left-hand side of equation \eqref{eq:th:2.5:1} is continuous and monotone with respective to $v(x_i)$. Then a zero point exists and equation \eqref{eq:th:2.5:1} admits a unique solution $v: V\rightarrow \mathbb{R}$.

  By rewriting equation \eqref{eq:th:2.5:1} as
  \begin{align*}
    v(x_i)=
    \begin{cases}
      \frac{1}{ \sum\limits_{k=1}^m w_k\chi_k(x_i)   g_{k}(u,v(x_i))}
      \sum\limits_{k=1}^m w_k\chi_k(x_i)   g_{k}(u,v(x_i))\left(\frac{\max\limits_{x_j\in e_k}u(x_j)+\min\limits_{x_j\in e_k}u(x_j)}{2}\right),\\
      \qquad\qquad\qquad\qquad\qquad\qquad\qquad\quad \mbox{if } x_i\in V\backslash L,~ \sum\limits_{k=1}^m g_{k}(u,v(x_i))\neq 0,\\
      \sum\limits_{k=1}^m \chi_k(x_i)\frac{\max\limits_{x_j\in e_k}u(x_j)+\min\limits_{x_j\in e_k}u(x_j)}{2},\quad \mbox{if } x_i\in V\backslash L,~
      \sum\limits_{k=1}^m g_{k}(u,v(x_i))= 0, \\
      y_i,\quad x_i\in L,
    \end{cases}
  \end{align*}
  where
  \begin{align*}
    g_{k}(u,v(x_i))=
    \begin{cases}
      0,\quad \mbox{if } \max\limits_{x_j\in e_k}u(x_j)+\min\limits_{x_j\in e_k}u(x_j)-2v(x_i)=0,\\
      \left|\max\limits_{x_j\in e_k}u(x_j)+\min\limits_{x_j\in e_k}u(x_j)-2v(x_i)\right|^{p-2},\quad \mbox{otherwise},
    \end{cases}
  \end{align*}
  we further have $v\in X$.

Now we define a mapping $T: X\rightarrow X$ by $T(u)=v$.
 It is continuous and admits a fixed point $u$, which is also a solution of equation \eqref{AE}.
\end{proof}

The superiority of equation \eqref{AE} over equation \eqref{E} and the functional $F_{H}^{con}$ lies in the computational efficiency. It can  be solved with fixed-point iteration
\begin{align*}
  \begin{split}
  u^{k+1}(x_i)=u^{k}(x_i)+{\tau}
  \left(\sum\limits_{k=1}^m w_k\chi_k(x_i)
  \left|\max\limits_{x_j\in e_k}u^k(x_j)+\min\limits_{x_j\in e_k}u^k(x_j)-2u^k(x_i)\right|^{p-2}\right. \\
  \left.\left(\max\limits_{x_j\in e_k}u^k(x_j)+\min\limits_{x_j\in e_k}u^k(x_j)-2u^k(x_i)\right)
  \right),
  \end{split}
\end{align*}
for $x_i\in V\backslash L$, where $u^0$ is any initial guess of the solution and $\tau$ is the step size.
The Dirichlet boundary condition $u(x_i)=y_i$, $x_i\in L$ is posed at each step.
In the case $p=2$, we can further drop the step size by iterating
\begin{align}\label{eq:2.4:2}
\begin{split}
u^{k+1}&(x_i)=\\
&\begin{cases}
  \frac{1}{2\sum\limits_{k=1}^m w_k\chi_k(x_i)}
  \sum\limits_{k=1}^m w_k\chi_k(x_i)
  \left(\max\limits_{x_j\in e_k}u^k(x_j)+\min\limits_{x_j\in e_k}u^k(x_j)\right), & \mbox{if } x_i\in V\backslash L, \\
  y_i, & \mbox{if } x_i\in L.
\end{cases}
\end{split}
\end{align}
If
\begin{align*}
  u_0=
  \begin{cases}
    \min y_i, & \mbox{if } x_i\in V\backslash L, \\
    y_i, & \mbox{if } x_i\in L,
  \end{cases}
  \quad \left(\mbox{or }
  u_0=
  \begin{cases}
    \max y_i, & \mbox{if } x_i\in V\backslash L, \\
    y_i, & \mbox{if } x_i\in L,
  \end{cases}
  \right)
\end{align*}
the scheme is bounded and monotone.
Namely,
$\min y_i\leq u^k(x_i)\leq \max y_i$
and
$u^{k+1}(x_i)\geq u^k(x_i)$ (or $u^{k+1}(x_i)\leq u^k(x_i)$)
for any $k\geq 0$ and $x_i\in V$.
The convergence of \eqref{eq:2.4:2} follows.

\section{Numerical experiments}

In this section, we discuss the numerical performance of the proposed simplified hypergraph $p$-Laplacian equation \eqref{AE} for data interpolation and semi-supervised learning.
We focus on the case $p=2$, which is commonly used in practice.
All experiments are performed using MATLAB on a desktop equipped with an Intel Core i7 3.20 GHz CPU.

\subsection{Data interpolation in 1D}
Let $n=1280$ and $V=\{x_i\}_{i=1}^n$ be random numbers on the interval $[0,1]$ that follow the standard uniform distribution.
We assume that $6$ of the points are labeled (denoted by red circles in Figure \ref{fig:1Dk}). The goal is to interpolate the remaining points.

A ${k_n}$-nearest neighbor graph $G$ with vertex set $V$ can be constructed.
For any vertices $x_i,x_j\in V$, we connect them by an edge $e_{i,j}$ if $x_j$ is among the ${k_n}$-nearest neighbors of $x_i$, denoted by $x_j\stackrel{k_n}{\sim} x_i$.
We also connect them if $x_i\stackrel{k_n}{\sim} x_j$ for the sake of symmetry.
The constant weight $w_{i,j}=1$ is adopted for edge $e_{i,j}$.
Alternatively, we can construct a ${k_n}$-nearest neighbor hypergraph $H$ with the vertex set $V$.
For every vertex $x_i\in V$, we define a hyperedge
\begin{equation*}
  e_i=\{x_j\in V: x_j\stackrel{k_n}{\sim} x_i \}.
\end{equation*}
Weight $w_i=1$ is assigned to every hyperedge $e_i$, $i=1,\cdots,n$.

The interpolation problem becomes the semi-supervised learning on $G$ or $H$, which can be solved by equation \eqref{AE} (i.e., iteration scheme \eqref{eq:2.4:2}).
We also consider the graph $p$-Laplacian $F_{G}^{con}$, the hypergraph $p$-Laplacian $F_{CE}^{con}$ and $F_{H}^{con}$ for comparison, see \eqref{eq:1.0}--\eqref{eq:1.1} for their definitions.
$F_{G}^{con}$ is solved  by the algorithm of \cite{flores2022analysis}.
We utilize a fixed-point scheme for $F_{CE}^{con}$, which is as follows
\begin{align}\label{eq:2.5:1}
\begin{split}
u^{k+1}(x_i)=
\begin{cases}
  \frac{1}{\sum\limits_{l=1}^m w_l\chi_l(x_i)|e_l|}
  \sum\limits_{l=1}^m w_l\chi_l(x_i)
  \sum_{x_j\in e_l}u^k(x_j), & \mbox{if } x_i\in V\backslash L, \\
  y_i, & \mbox{if } x_i\in L.
\end{cases}
\end{split}
\end{align}
$F_{H}^{con}$ is solved by the stochastic PDHG \cite{shi2024hypergraph}.

To compare the interpolation result and the computation time of different algorithms, we run four algorithms for a sufficiently long time to obtain the ``true solutions" $u^*$ (shown in Figure \ref{fig:1Dk}).
The running time of hypergraph models with respect to the relative $\ell^2$ error
\begin{equation*}
  \frac{\|u-u^*\|_{\ell^2}}{\|u^*\|_{\ell^2}}
\end{equation*}
is then plotted in Figure \ref{fig:time_1D}.

As illustrated in Figure \ref{fig:1Dk}, all four algorithms effectively interpolate the data when ${k_n} = 9$.
However, as ${k_n}$ increases, notable differences emerge.
Since the solution at each unlabeled point is the average of its neighbors,
$F_{G}^{con}$ develops spikes at the labeled points for large $k_n$.
The phenomenon can also be explained from the fact that the variational consistency between $F_{G}^{con}$ and the continuum $p$-Laplacian $\mathcal{F}^{con}$ holds only when the connection radius is sufficiently small.
$F_{CE}^{con}$ exhibits similar spiking behavior as it is essentially the graph Laplacian.
In contrast, $F_{H}^{con}$ effectively suppresses spiky solutions.
 As an approximation of $F_{H}^{con}$,  equation \eqref{AE} produces solutions with a similar structure.
The difference between the two is that $F_{H}^{con}$ gives better interpolation results near the labeled points (see the 3rd and 4th labeled points), while equation \eqref{AE} provides smoother results (see the case ${k_n}=72$).

Figure \ref{fig:time_1D} shows the computational cost of different algorithms. Equation \eqref{AE} outperforms $F_{H}^{con}$ by a large margin and is even better than $F_{CE}^{con}$.
Notably, its running time decreases as the parameter ${k_n}$ increases.
This is due to the fact that a larger ${k_n}$ increases the cardinality of vertices, which in turn accelerates the convergence of equation \eqref{AE}.
The running time for large ${k_n}=72$ is comparable to that of the graph model $F_{G}^{con}$ with a recent algorithm \cite{flores2022analysis} ($\approx0.17s$).

\begin{figure}[tbp]
\centering
\begin{minipage}[t]{0.24\linewidth}
\centering
\includegraphics[width=1\textwidth]{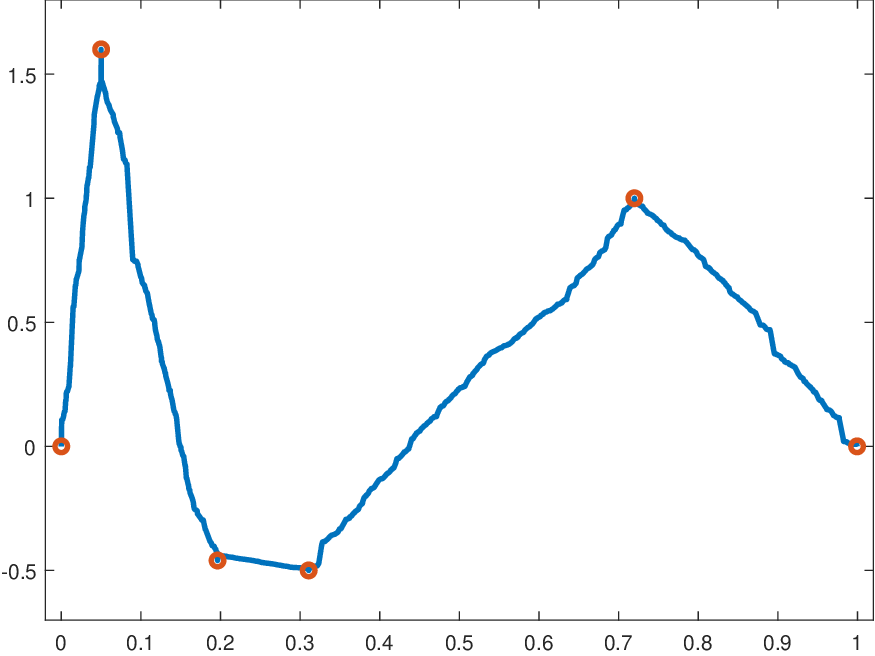}\\
\includegraphics[width=1\textwidth]{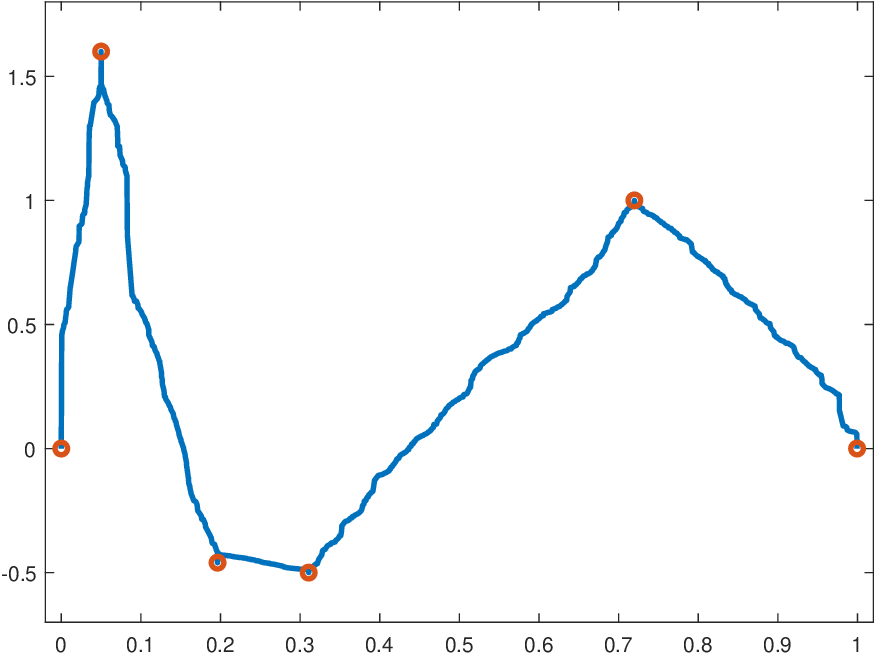} \\
\includegraphics[width=1\textwidth]{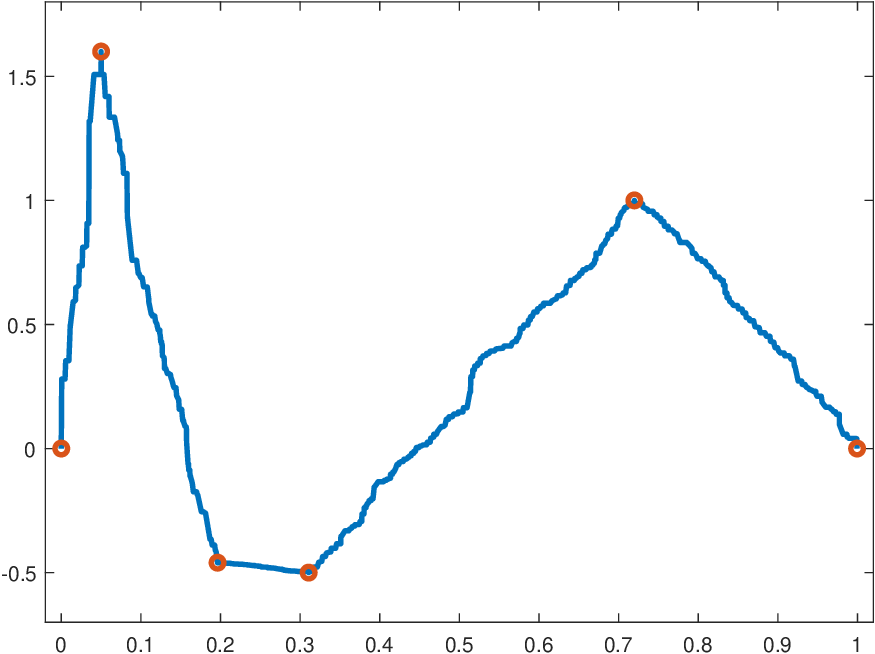}\\
\includegraphics[width=1\textwidth]{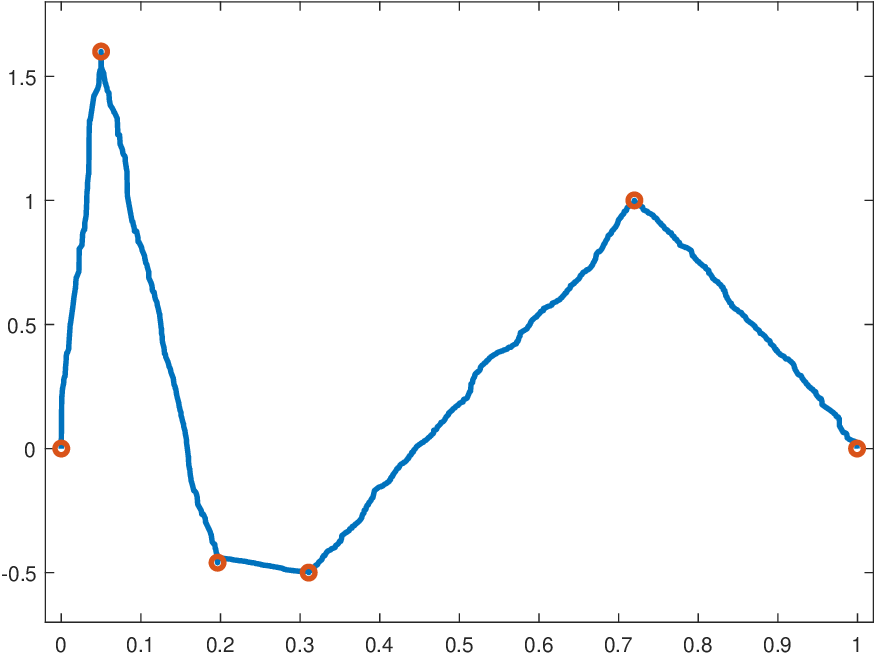}\\
\centerline{\footnotesize\emph{${k_n}=9$}}
\end{minipage}
\centering
\begin{minipage}[t]{0.24\linewidth}
\centering
\includegraphics[width=1\textwidth]{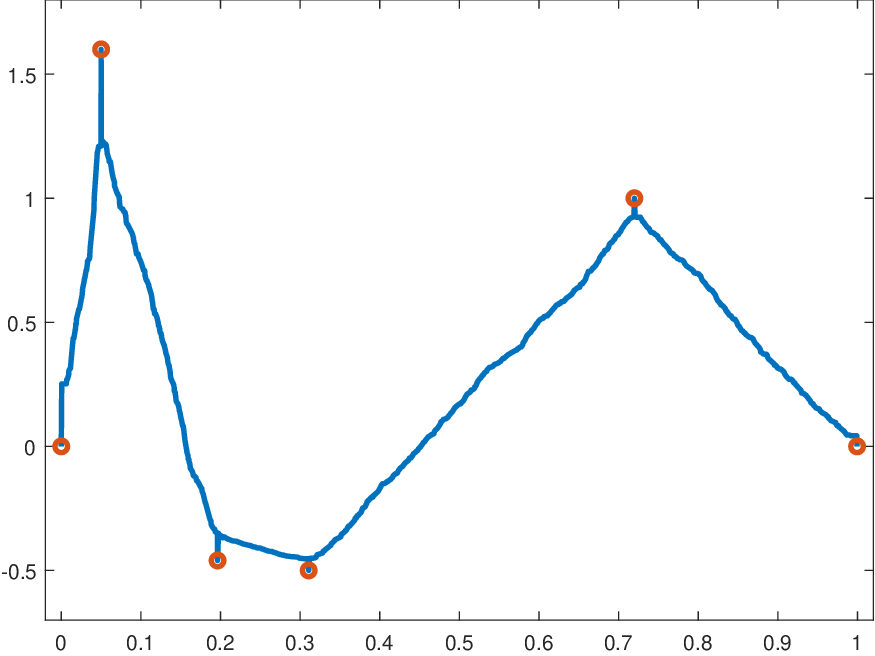}\\
\includegraphics[width=1\textwidth]{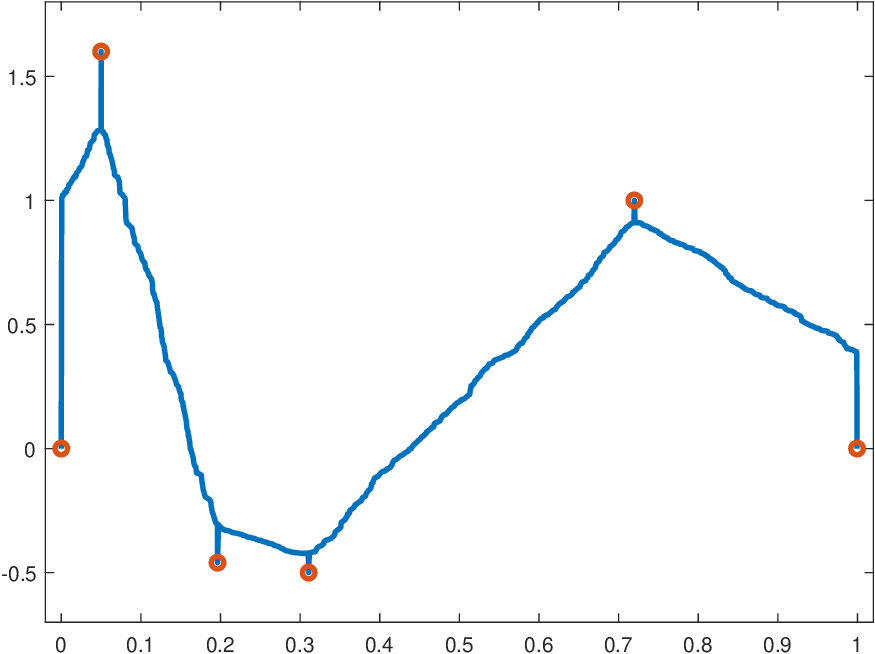} \\
\includegraphics[width=1\textwidth]{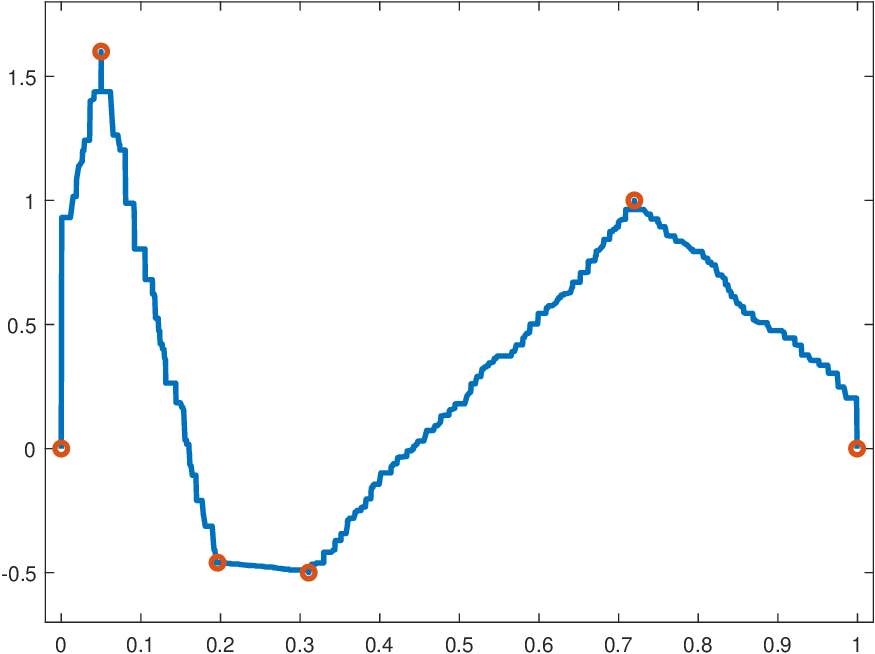}\\
\includegraphics[width=1\textwidth]{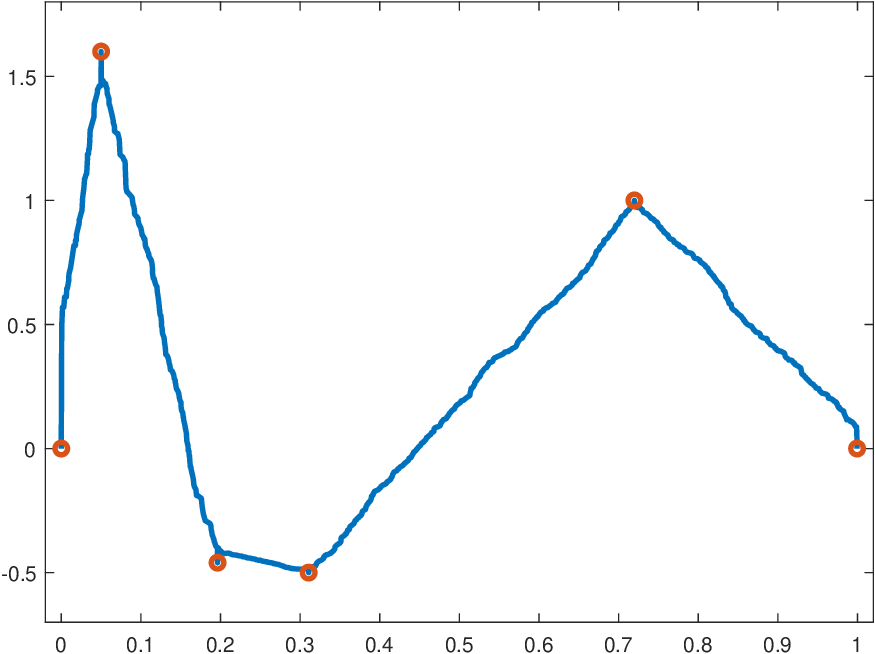}\\
\centerline{\footnotesize\emph{${k_n}=18$}}
\end{minipage}
\centering
\begin{minipage}[t]{0.24\linewidth}
\centering
\includegraphics[width=1\textwidth]{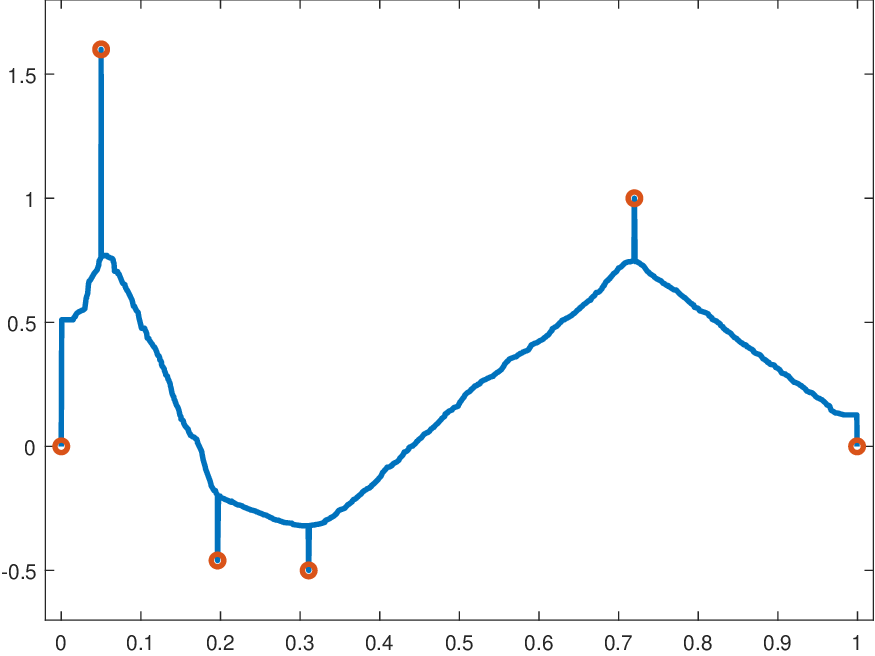}\\
\includegraphics[width=1\textwidth]{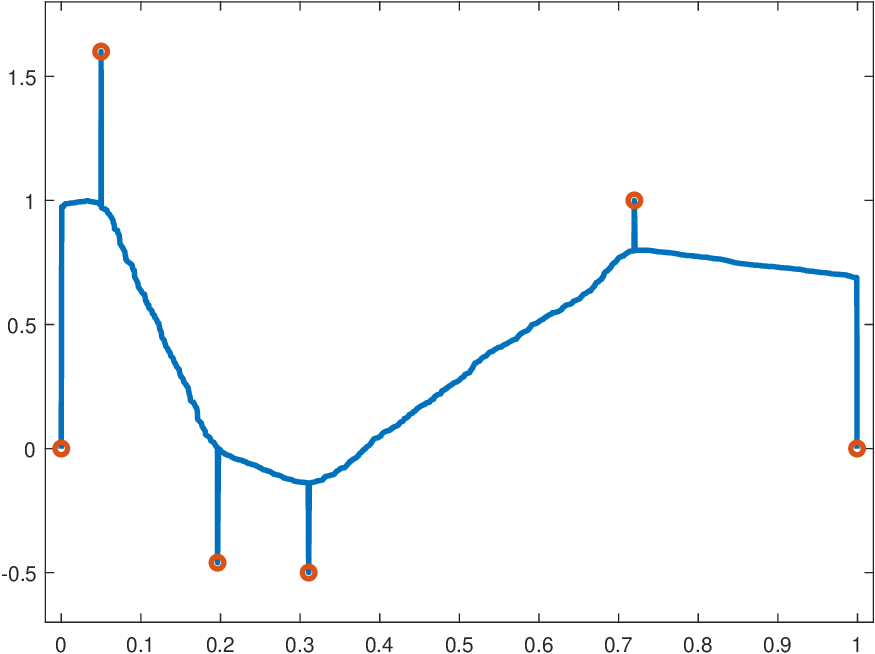} \\
\includegraphics[width=1\textwidth]{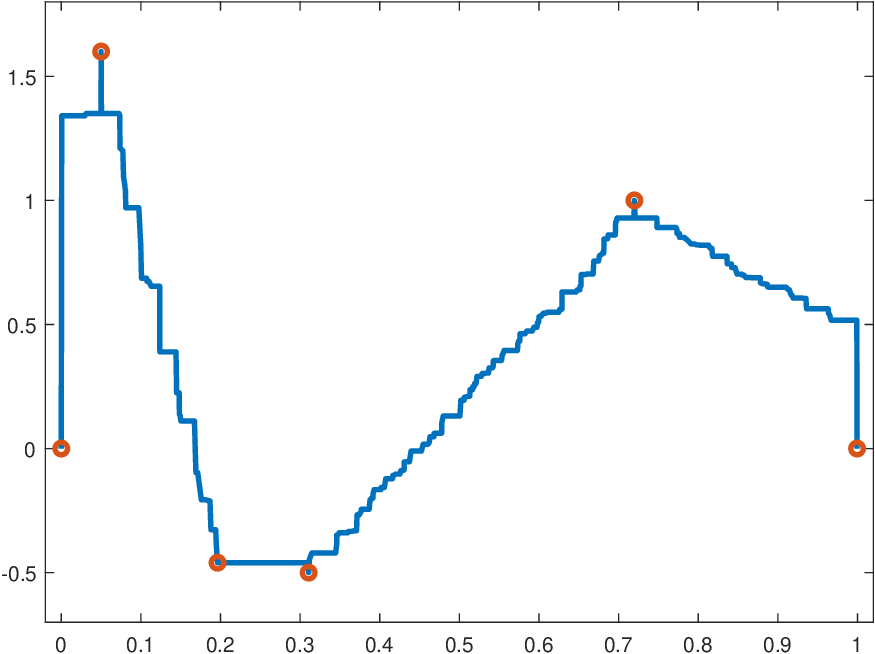}\\
\includegraphics[width=1\textwidth]{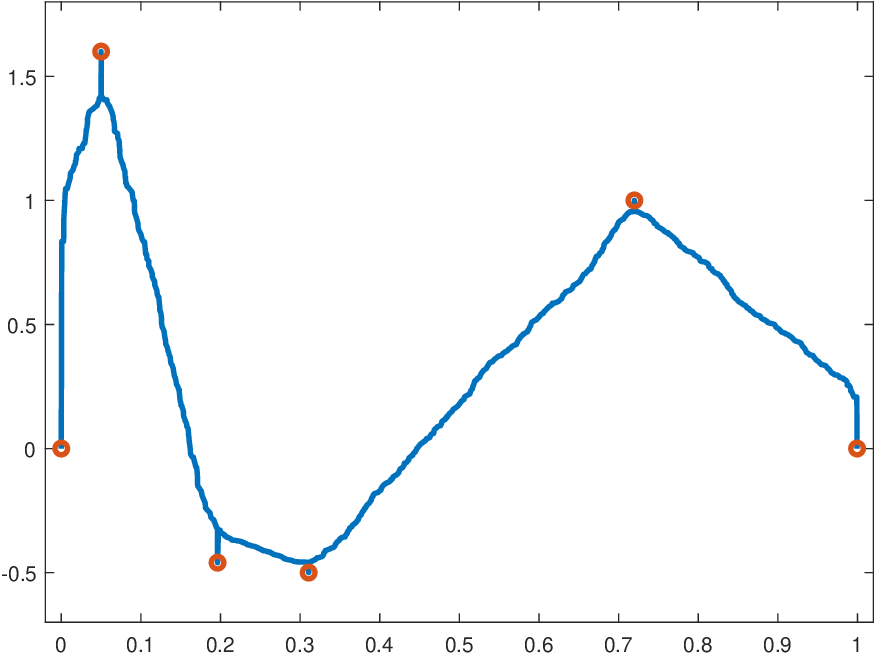}\\
\centerline{\footnotesize\emph{${k_n}=36$}}
\end{minipage}
\centering
\begin{minipage}[t]{0.24\linewidth}
\centering
\includegraphics[width=1\textwidth]{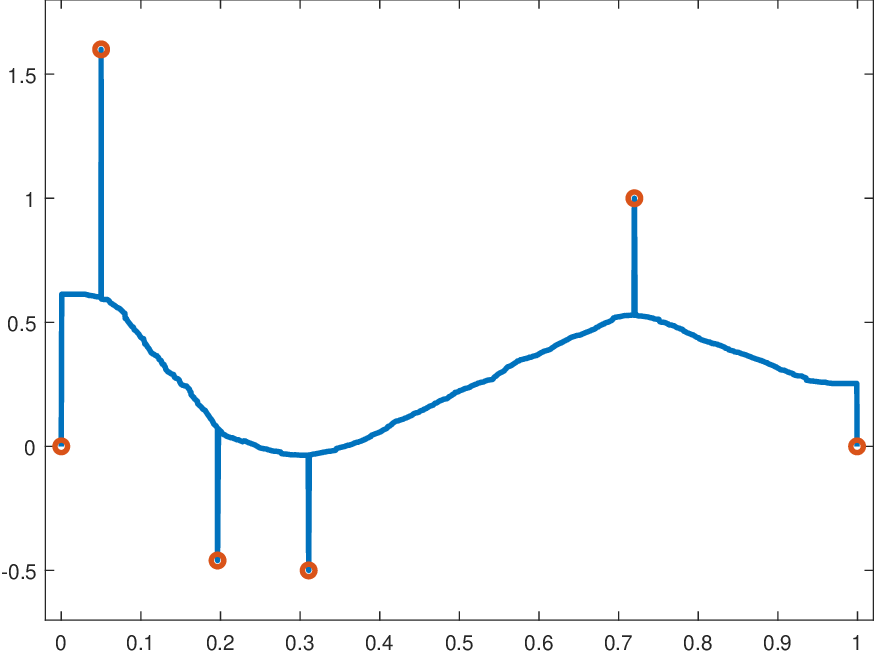}\\
\includegraphics[width=1\textwidth]{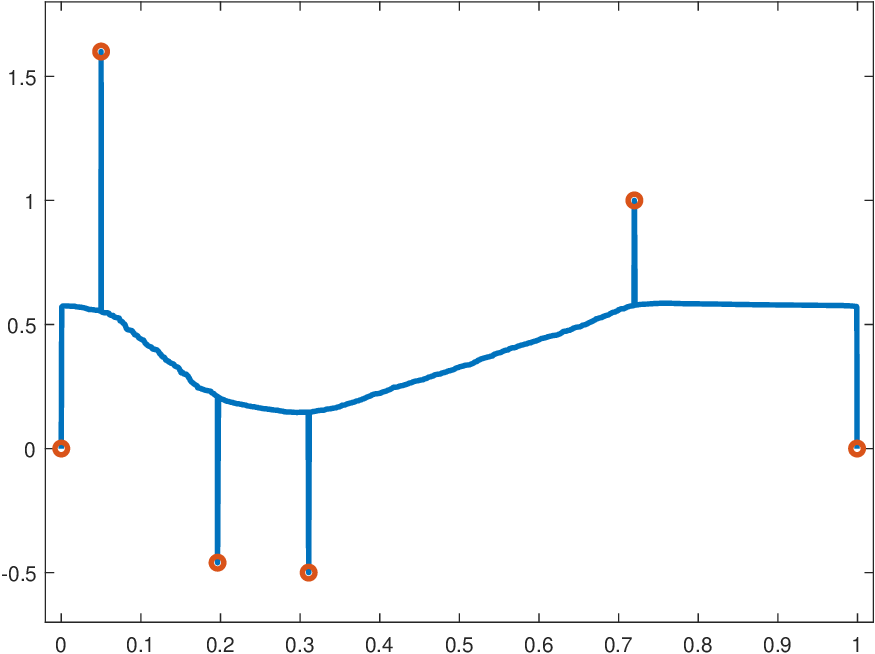} \\
\includegraphics[width=1\textwidth]{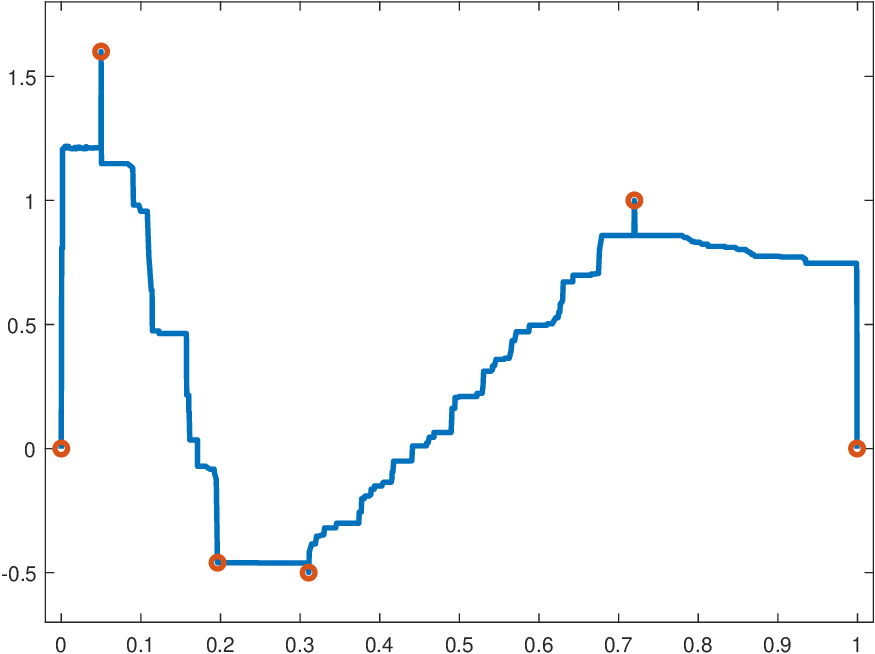}\\
\includegraphics[width=1\textwidth]{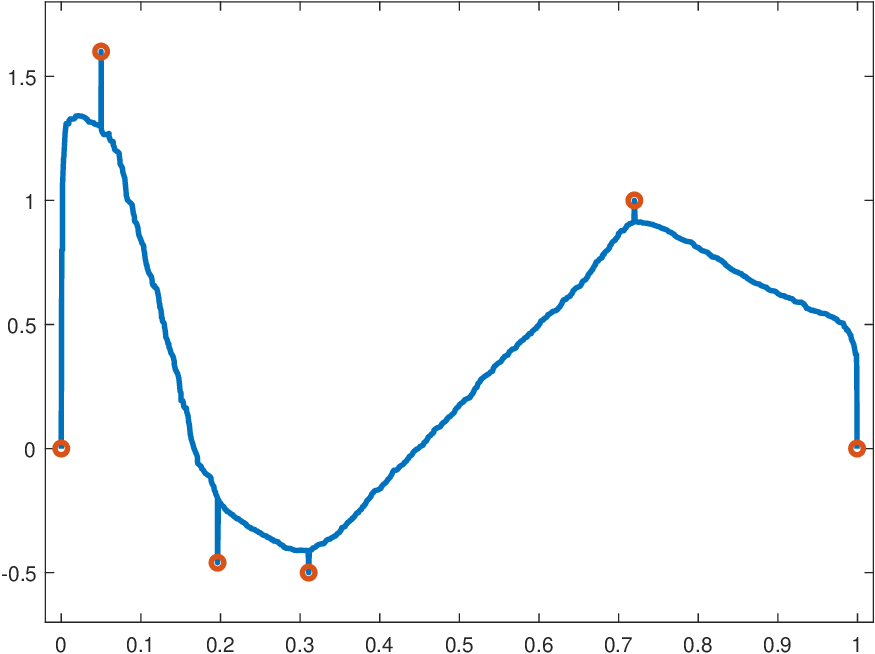}\\
\centerline{\footnotesize\emph{${k_n}=72$}}
\end{minipage}
\caption{Results of graph Laplacian and hypergraph Laplacian for different ${k_n}$. From top to bottom: $F_{G}^{con}$, $F_{CE}^{con}$, $F_{H}^{con}$, equation \eqref{AE}.}
\label{fig:1Dk}
\end{figure}

\begin{figure}[tbp]
    \centering
\begin{minipage}[t]{0.32\linewidth}
\centering
\includegraphics[width=1\textwidth]{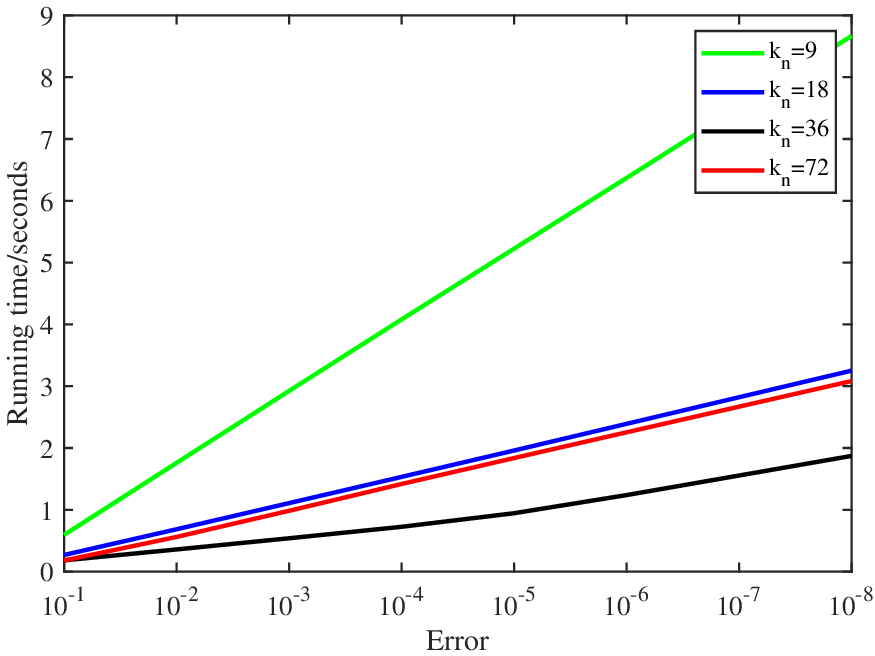}\\
\centerline{\footnotesize\emph{$F_{CE}^{con}$}}
\end{minipage}
\begin{minipage}[t]{0.32\linewidth}
\centering
\includegraphics[width=1\textwidth]{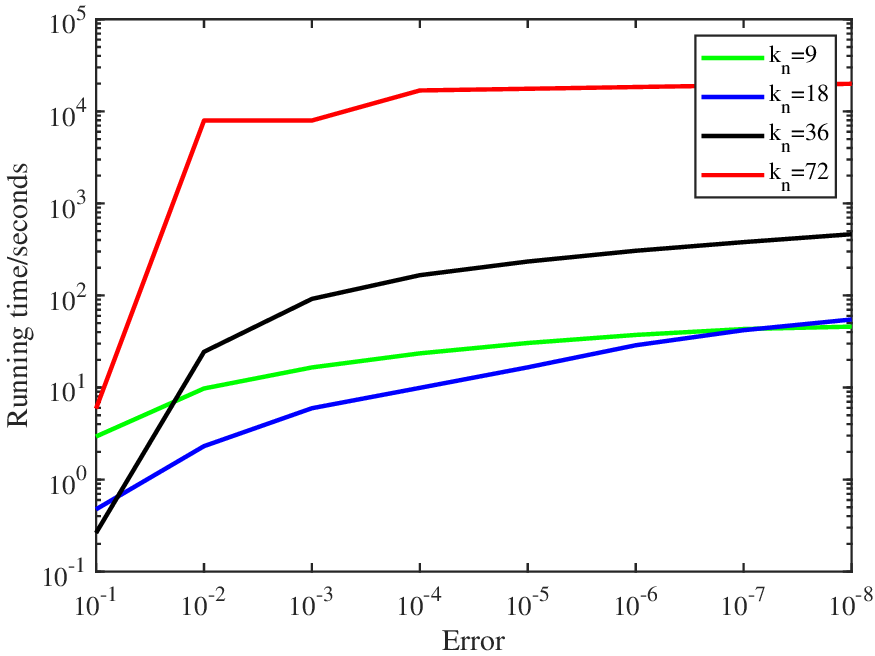}\\
\centerline{\footnotesize\emph{$F_{H}^{con}$}}
\end{minipage}
\begin{minipage}[t]{0.32\linewidth}
\centering
\includegraphics[width=1\textwidth]{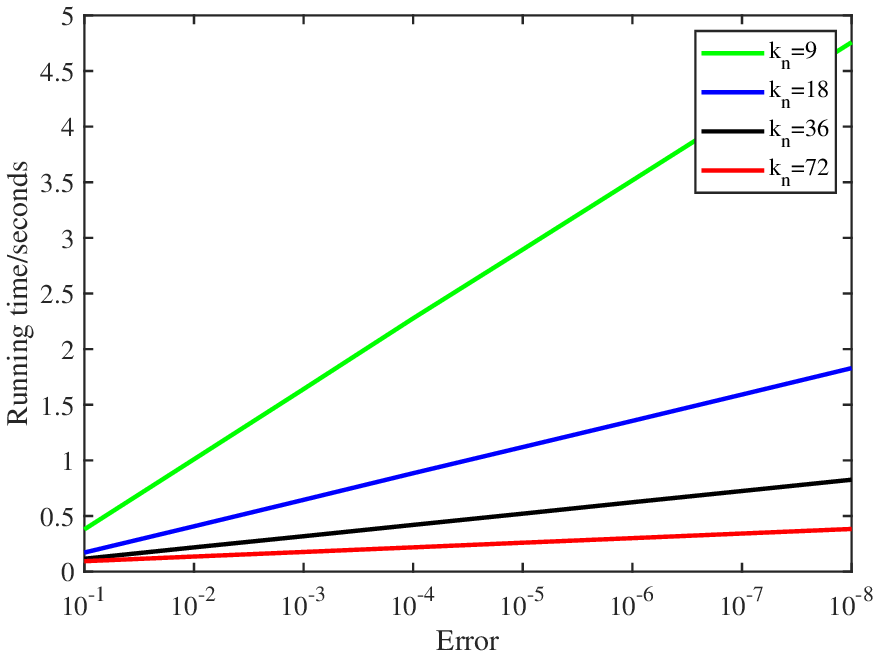}\\
\centerline{\footnotesize\emph{Equation \eqref{AE}}}
\end{minipage}
\caption{The running time of different methods with respect to the relative $\ell_2$ error.}
 \label{fig:time_1D}
\end{figure}

\subsection{Semi-supervised learning}
In the rest of this section, we consider the performance of $F_{CE}^{con}$, $F_{H}^{con}$, and equation \eqref{AE} for semi-supervised learning on some real-world datasets summarized in Table \ref{tab:dataset}.
Both Mushroom and Covertype come from the UCI repository.
Covertype(4,5) and Covertype(6,7) are derived from Covertype by selecting two classes, (4,5) and (6,7) respectively.
They are provided by the first author of \cite{hein2013total}.
All three datasets contain only categorical features. We construct a hypergraph $H=(V,E,W)$ from a given dataset as follows.
For each feature and each category of the feature, we construct a hyperedge $e\in E$ by joining all vertices that belong to the same feature and category.
Duplicate hyperedges are removed.
Cora, Pubmed, and DBLP \cite{sen2008collective} are hypergraph datasets, in which vertices are documents and hyperedges follow from co-citation or co-authorship.
Isolated vertices are removed from them.
The weight $w=1$ is used for all experiments.

Algorithm \ref{alg1} provides the details of semi-supervised learning on hypergraphs with $F_{CE}^{con}$, $F_{H}^{con}$, and equation \eqref{AE}.

\begin{table}
  \centering
  \caption{Datasets for semi-supervised learning.}
    \begin{tabular}{lllllll}
    \toprule
    Dataset & Mushroom & Covertype & Covertype & Cora & Pubmed & DBLP\\
     & & (4,5) & (6,7) & citation & citation & authorship \\
    \midrule
    $|V|$ & 8124 & 12240 & 37877 & 2708 & 3312 & 43413 \\
    $|E|$ & 110 & 104 & 123 & 1579 & 1079 & 22535 \\
    \#labels & 2 & 2 & 2 & 7 & 3 & 6 \\
    \bottomrule
    \end{tabular}%
  \label{tab:dataset}%
\end{table}%

\begin{algorithm}
  \caption{Semi-supervised learning with hypergraph $p$-Laplacian.}
  \label{alg1}
  \begin{algorithmic}
  \REQUIRE A dataset represented by a hypergraph $H=(V,E,W)$, a
  training set $\{(x_i,y_i): x_i\in L\subset V, y_i\in\{1,2,\cdots l\}\}$.
  \FOR { $s=1:l$}
  \STATE Find the constraint: For any $x_i\in L$, let
  \begin{align*}
    u_s(x_i)=
    \begin{cases}
      1, & \mbox{if } y_i=s, \\
      -1, & \mbox{otherwise}.
    \end{cases}
  \end{align*}
  \STATE Interpolate $u_s$ on $V\backslash L$ by minimizing $F_{CE}^{con}$, $F_{H}^{con}$, or solving equation \eqref{AE}.
  \ENDFOR
  \STATE For any $x_i\in V\backslash L$, let
  \begin{equation*}
    y_i=\arg\max_{1\leq s\leq l}u_s(x_i),\quad u(x_i)=y_i.
  \end{equation*}
  \RETURN $u$.
  \end{algorithmic}
\end{algorithm}

The stochastic PDHG algorithm is not suitable for $F_{H}^{con}$ when hyperedges have large cardinality (see Figure \ref{fig:time_1D}). We adopt the subgradient descent algorithm proposed in \cite{zhang2017re} for minimizing $F_{H}^{con}$.
The step size is empirically chosen as
\begin{equation}\label{eq:4.1}
  \tau(k)=\frac{1}{(k+1)^{\min(0.16k/10^5,1)}}
\end{equation}
to speed up the algorithm, as suggested by the authors.
In this setting, it is not easy to find a robust stopping criterion for the algorithm since it is not convergent in general.
This is also the reason that we do not utilize it in the previous subsection.
Instead, we manually select the iteration number within $[0,10,000]$ that achieves the minimum energy for each dataset and each labeling rate.
While $F_{CE}^{con}$ and our equation \eqref{AE} are easy to implement. A simple stopping criterion like
$\max |u^{k-1}- u^{k}|\leq \varepsilon$
is feasible.
We find that $\varepsilon=10^{-3}$ for $F_{CE}^{con}$ and $\varepsilon=10^{-2}$ for equation \eqref{AE} give the expected classification result.

\begin{figure}
    \centering
\begin{minipage}[t]{0.32\linewidth}
\centering
\includegraphics[width=1\textwidth]{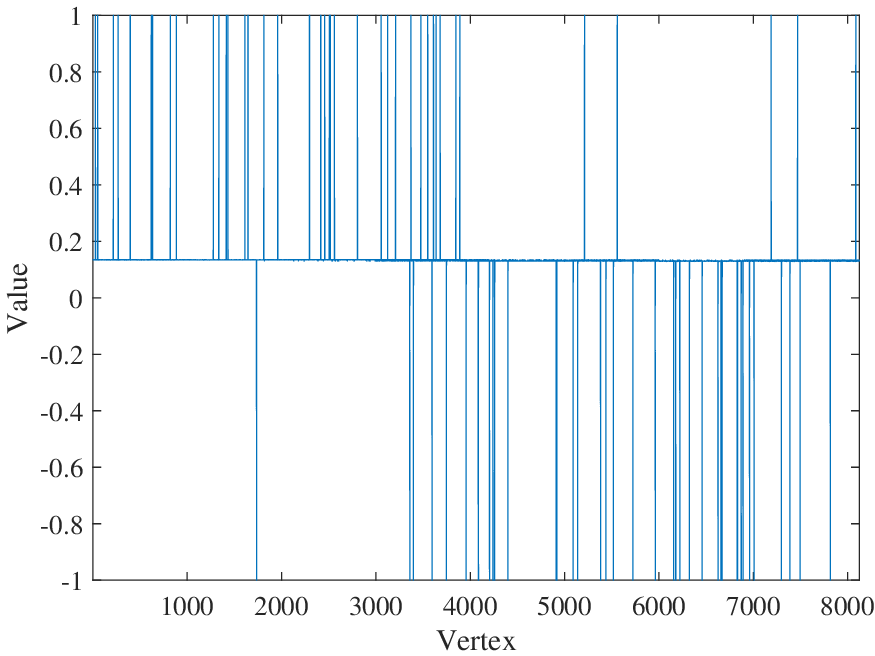}\\
\includegraphics[width=1\textwidth]{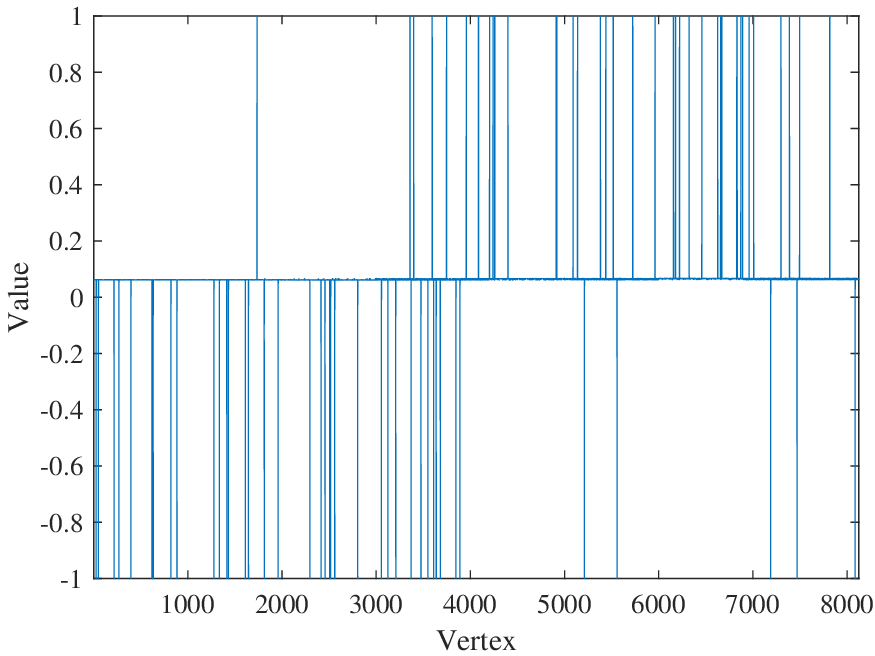}\\
\centerline{\footnotesize\emph{$F_{CE}^{con}$}}
\end{minipage}
\begin{minipage}[t]{0.32\linewidth}
\centering
\includegraphics[width=1\textwidth]{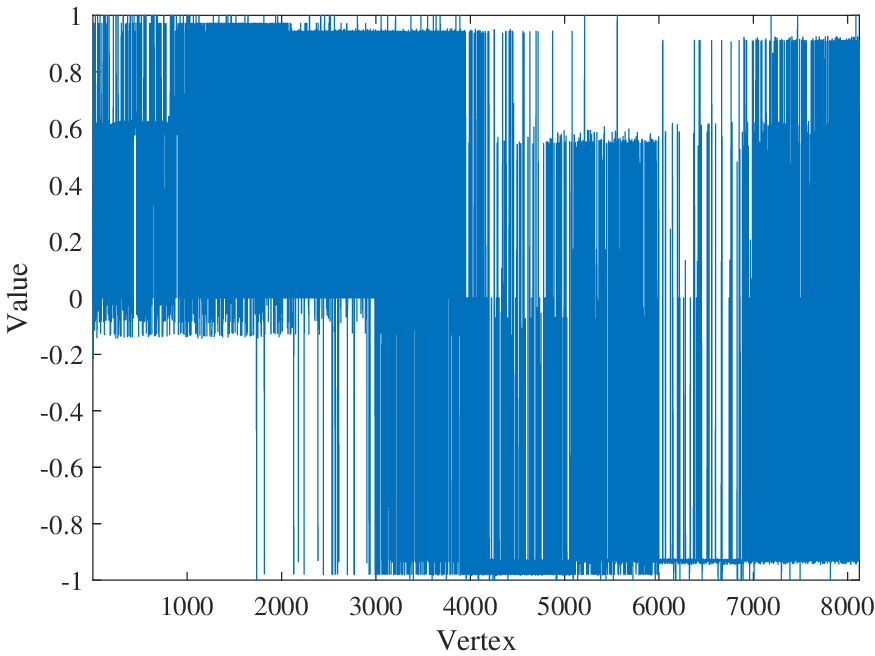}\\
\includegraphics[width=1\textwidth]{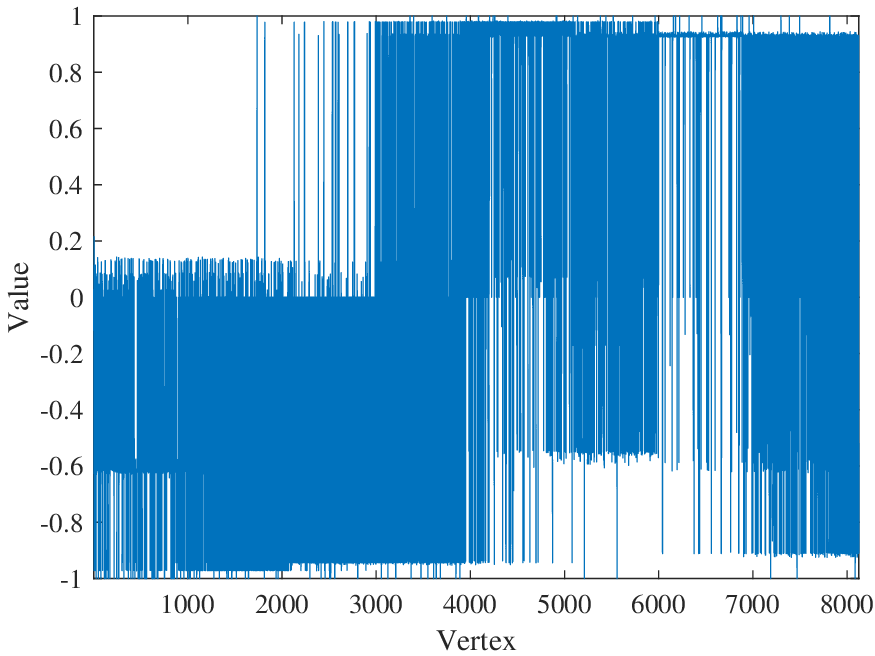}\\
\centerline{\footnotesize\emph{$F_{H}^{con}$}}
\end{minipage}
\begin{minipage}[t]{0.32\linewidth}
\centering
\includegraphics[width=1\textwidth]{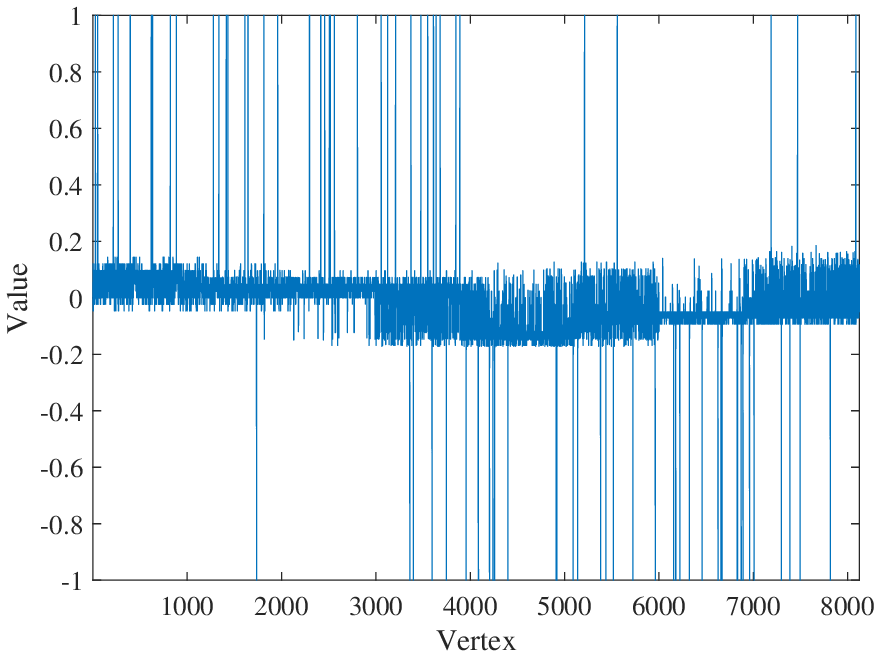}\\
\includegraphics[width=1\textwidth]{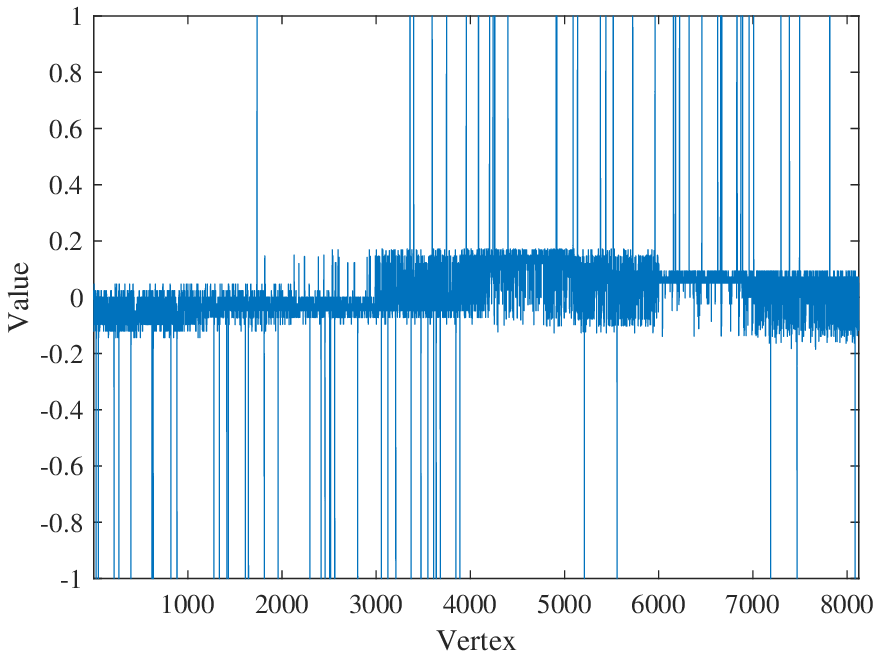}\\
\centerline{\footnotesize\emph{Equation \eqref{AE}}}
\end{minipage}
\caption{Solutions of different algorithms for the dataset Mushroom with 1\% labeling rate. The first row and second row are the solution $u_1$ and $u_2$ in Algorithm \ref{alg1} respectively. Vertices with values $\pm 1$ are the labeled points.}
 \label{fig:solution}
\end{figure}

\begin{table}[htbp]
  \setlength\tabcolsep{4.5pt}
  \centering
  \caption{Classification error and standard deviation of $F_{CE}^{con}$, $F_{H}^{con}$, and equation \eqref{AE}.}
    \begin{tabular}{lllllll}
    \toprule
    Dataset & Mushr. & Covert45 & Covert67 & Cora  & Pubmed & DBLP \\\midrule\midrule
    Labeling rate & 1\%   & 0.64\% & 0.2\% & 10\%  & 10\%  & 10\% \\\midrule
    $F_{CE}^{con}$ & 49.2$\pm$1.9 & 22.3$\pm$0.0 & 46.5$\pm$2.3 & 22.0$\pm$1.6 & 21.3$\pm$1.1 & 19.7$\pm$0.3 \\
    $F_{H}^{con}$ & 4.7$\pm$2.6 & 7.9$\pm$1.6 & 6.0$\pm$7.2 & 22.7$\pm$1.4 & 18.0$\pm$0.5 & 21.0$\pm$0.3 \\
    Equation (13) & 4.2$\pm$1.9 & 0.5$\pm$0.9 & 0.7$\pm$0.3 & 22.1$\pm$1.3 & 17.0$\pm$0.5 & 19.9$\pm$0.4 \\\midrule\midrule
    Labeling rate & 0.5\%   & 0.32\% & 0.1\% & 5\%   & 5\%   & 5\% \\\midrule
    $F_{CE}^{con}$ & 49.4$\pm$1.9 & 22.4$\pm$0.0 & 40.7$\pm$15.4 & 31.2$\pm$6.0 & 27.6$\pm$1.8 & 24.0$\pm$0.3 \\
    $F_{H}^{con}$ & 8.2$\pm$3.7 & 14.9$\pm$2.1 & 22.0$\pm$12.1 & 32.8$\pm$6.1 & 22.6$\pm$1.3 & 25.1$\pm$0.3 \\
    Equation (13) & 6.0$\pm$2.6 & 3.5$\pm$3.1 & 1.8$\pm$1.0 & 29.0$\pm$4.4 & 19.5$\pm$0.7 & 23.6$\pm$0.4 \\\midrule\midrule
    Labeling rate & 0.25\%   & 0.16\% & 0.05\% & 2.5\% & 2.5\% & 2.5\% \\\midrule
    $F_{CE}^{con}$ & 49.5$\pm$1.9 & 22.4$\pm$0.0 & 49.1$\pm$4.3 & 51.1$\pm$7.3 & 32.1$\pm$2.2 & 29.5$\pm$1.0 \\
    $F_{H}^{con}$ & 19.4$\pm$8.6 & 18.8$\pm$2.1 & 38.5$\pm$9.4 & 51.9$\pm$4.3 & 31.5$\pm$7.1 & 30.0$\pm$0.8 \\
    Equation (13) & 8.4$\pm$2.8 & 6.5$\pm$5.7 & 9.2$\pm$6.9 & 41.2$\pm$3.8 & 23.7$\pm$2.2 & 27.0$\pm$0.3 \\
    \bottomrule
    \end{tabular}%
  \label{tab:SSL}%
\end{table}%

To compare the classification accuracy of different algorithms, we randomly select some points from the dataset as the training set and repeat the experiment 10 times.
The classification error and the standard deviation are summarized in Table \ref{tab:SSL}.
For datasets Mushroom and Covertype, the minimizer of $F_{CE}^{con}$ is approximately a constant on the unlabeled points, resulting in poor classification accuracy (see the first column of Figure \ref{fig:solution}).
This is the spiking phenomenon we observed in Figure \ref{fig:1Dk}. The hypergraph algorithms $F_{H}^{con}$ and equation \eqref{AE} avoid this problem (see the second and third columns of Figure \ref{fig:solution}).
It can also be seen from Figure \ref{fig:solution} that $F_{H}^{con}$ suppresses spiky solutions better than equation \eqref{AE}, coinciding with the result of Figure \ref{fig:1Dk}.

By comparing the results of $F_{CE}^{con}$ and equation \eqref{AE} on Datasets Cora, Pubmed, and DBLP, we find that the latter has better classification performance in most of the cases. However, as the labeling rate increases, the advantage of the latter decreases or even disappears (see the datasets Cora and DBLP with a labeling rate of 10\%). This suggests that the hypergraph algorithm is more suitable for problems with low labeling rates.
By comparing the results of $F_{H}^{con}$ and equation \eqref{AE}, we are surprised to find that
equation \eqref{AE} achieves better classification accuracy, even though it comes from an approximation of $F_{H}^{con}$  and  cannot suppress spiky solutions as well as $F_{H}^{con}$.

The main contribution of equation \eqref{AE} lies in the computational efficiency and the stability.
The average running times of $F_{CE}^{con}$, $F_{H}^{con}$, and equation \eqref{AE} on all experiments are about 23.7, 947, and 3.9 seconds, respectively.
The difference between the results of $F_{CE}^{con}$ and equation \eqref{AE} comes from the iteration number, i.e., the stopping criterion.
$F_{H}^{con}$ requires a large computational cost, since each experiment requires 10,000 iterations and the calculation of the corresponding energy.
The numerical scheme \eqref{eq:2.4:2} for equation \eqref{AE} is convergent and does not involve any parameters, thus avoiding this problem.

Let us consider separately the largest dataset DBLP, which is also the only dataset in Table \ref{tab:dataset} that the subgradient descent algorithm for $F_{H}^{con}$ converges.
Figure \ref{fig:energy_DBLP} shows the energy of $F_{H}^{con}$ with respect to the iteration number, from which we see that the iteration number $(700, 1,100, 1,500)$ works for the labeling rate $(10\%, 5\%, 2.5\%)$.
Table \ref{tab:time_SSL} shows the average running time of the three algorithms over 10 runs.
For this particular dataset, the running time is also greatly reduced by equation \eqref{AE}.

\begin{figure}[tbp]
    \centering
\begin{minipage}[t]{0.32\linewidth}
\centering
\includegraphics[width=1\textwidth]{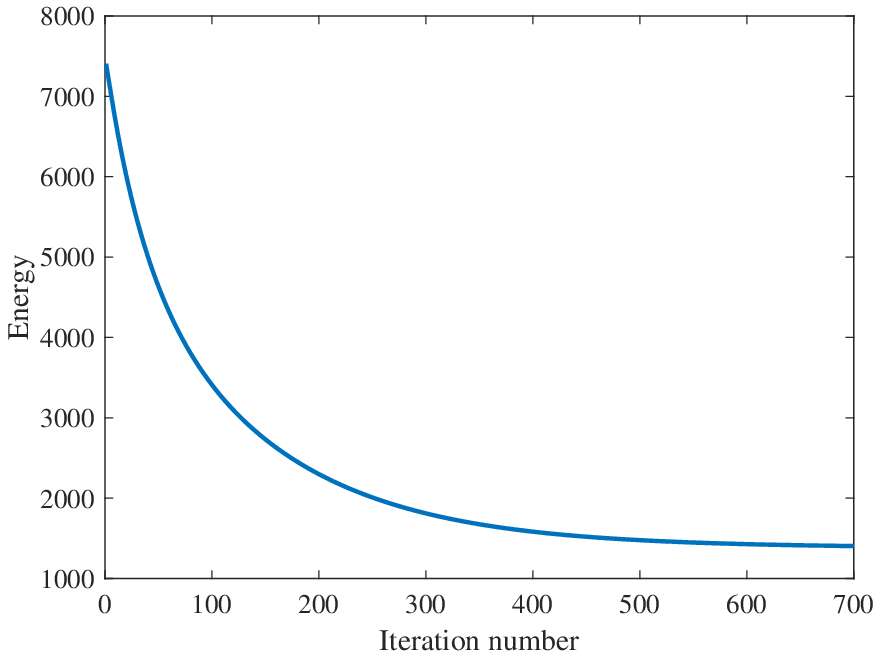}\\
\centerline{\footnotesize\emph{Labeling rate: 10\%}}
\end{minipage}
\begin{minipage}[t]{0.32\linewidth}
\centering
\includegraphics[width=1\textwidth]{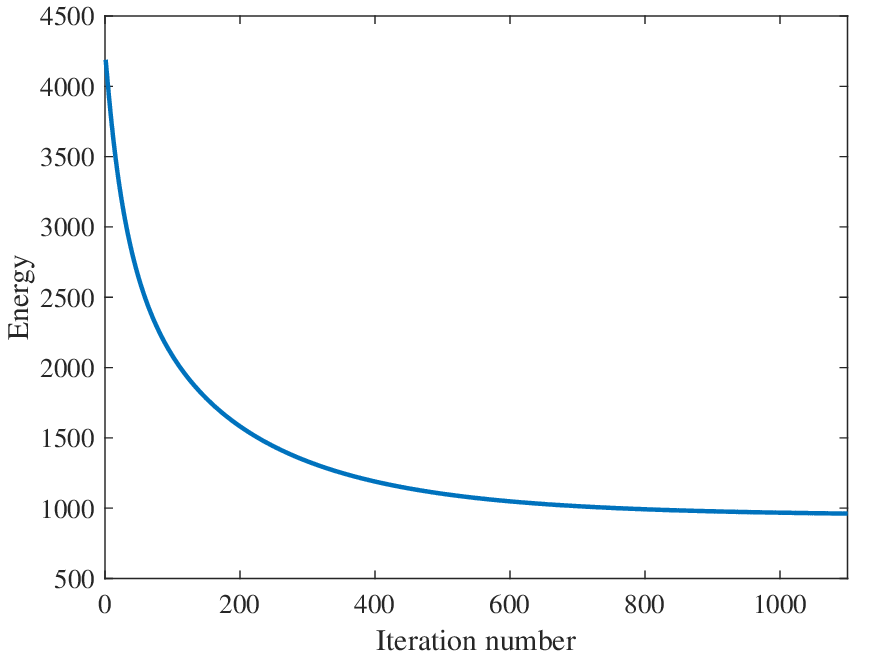}\\
\centerline{\footnotesize\emph{Labeling rate: 5\%}}
\end{minipage}
\begin{minipage}[t]{0.32\linewidth}
\centering
\includegraphics[width=1\textwidth]{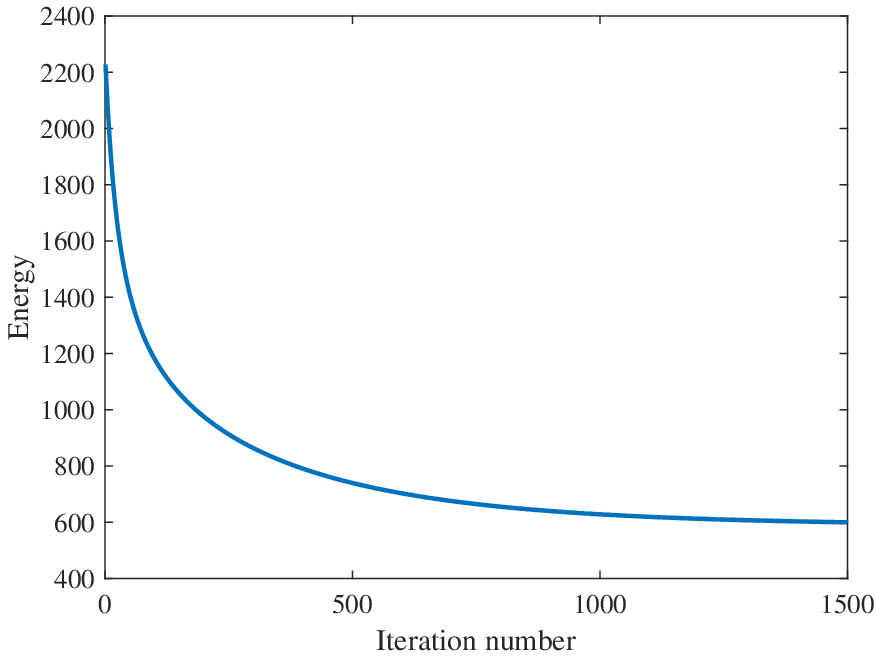}\\
\centerline{\footnotesize\emph{Labeling rate: 2.5\%}}
\end{minipage}
\caption{The energy of $F_{H}^{con}$ for the dataset DBLP.}
 \label{fig:energy_DBLP}
\end{figure}

\begin{table}[tbp]
  \centering
  \caption{The average running time (seconds) of $F_{CE}^{con}$, $F_{H}^{con}$, and equation \eqref{AE} for the dataset DBLP.}
    \begin{tabular}{llll}
    \toprule
    Labeling rate & 10\%  & 5\% & 2.5\% \\
    \midrule
    $F_{CE}^{con}$ & 262.4     & 429.9     & 584.1 \\
    $F_{H}^{con}$ & 116.1     & 124.9     & 137.3 \\
    Equation (13) & 19.1     & 20.2     & 20.6 \\
    \bottomrule
    \end{tabular}%
  \label{tab:time_SSL}%
\end{table}%

\section{Conclusion}
In this paper, a new hypergraph $p$-Laplacian equation, deduced from an approximation of the hypergraph $p$-Laplacian regularization, has been proposed for semi-supervised learning.
The unique solvability and the comparison principle of the equation have been established.
Numerical experiments have confirmed the effectiveness of the new equation. It not only suppresses spiky solutions and improves the classification accuracy, but also significantly reduces the computation time.

\section*{Acknowledgements}
The authors would like to thank the referees for the valuable comments and suggestions.
KS is supported by China Scholarship Council.
The authors acknowledge  support from DESY (Hamburg, Germany), a member of the Helmholtz Association HGF.

\section*{Declarations}
\textbf{Data Availability} Data will be made available on request.

\textbf{Conflict of interest} The authors have no relevant financial or non-financial interests to disclose.










%
%

\bibliographystyle{unsrt}
\bibliography{references}

\begin{thebibliography}{10}

\bibitem{zhang2020hypergraph}
Songyang Zhang, Shuguang Cui, and Zhi Ding.
\newblock Hypergraph-based image processing.
\newblock In {\em 2020 IEEE International Conference on Image Processing
  (ICIP)}, pages 216--220. IEEE, 2020.

\bibitem{shi2024continuum}
Kehan Shi and Martin Burger.
\newblock Continuum limit of $ p $-biharmonic equations on graphs.
\newblock {\em arXiv preprint arXiv:2404.19689}, 2024.

\bibitem{klamt2009hypergraphs}
Steffen Klamt, Utz-Uwe Haus, and Fabian Theis.
\newblock Hypergraphs and cellular networks.
\newblock {\em PLoS computational biology}, 5(5):e1000385, 2009.

\bibitem{patro2013predicting}
Rob Patro and Carl Kingsford.
\newblock Predicting protein interactions via parsimonious network history
  inference.
\newblock {\em Bioinformatics}, 29(13):i237--i246, 2013.

\bibitem{antelmi2021social}
Alessia Antelmi, Gennaro Cordasco, Carmine Spagnuolo, and Przemys{\l}aw Szufel.
\newblock Social influence maximization in hypergraphs.
\newblock {\em Entropy}, 23(7):796, 2021.

\bibitem{fazeny2023hypergraph}
Ariane Fazeny, Daniel Tenbrinck, and Martin Burger.
\newblock Hypergraph p-laplacians, scale spaces, and information flow in
  networks.
\newblock In {\em International Conference on Scale Space and Variational
  Methods in Computer Vision}, pages 677--690. Springer, 2023.

\bibitem{zhou2006learning}
Dengyong Zhou, Jiayuan Huang, and Bernhard Sch{\"o}lkopf.
\newblock Learning with hypergraphs: Clustering, classification, and embedding.
\newblock {\em Advances in neural information processing systems}, 19, 2006.

\bibitem{agarwal2006higher}
Sameer Agarwal, Kristin Branson, and Serge Belongie.
\newblock Higher order learning with graphs.
\newblock In {\em Proceedings of the 23rd international conference on Machine
  learning}, pages 17--24, 2006.

\bibitem{el2016asymptotic}
Ahmed El~Alaoui, Xiang Cheng, Aaditya Ramdas, Martin~J Wainwright, and
  Michael~I Jordan.
\newblock Asymptotic behavior of$\ell_p$-based laplacian
  regularization in semi-supervised learning.
\newblock In {\em Conference on Learning Theory}, pages 879--906. PMLR, 2016.

\bibitem{hein2013total}
Matthias Hein, Simon Setzer, Leonardo Jost, and Syama~Sundar Rangapuram.
\newblock The total variation on hypergraphs-learning on hypergraphs revisited.
\newblock {\em Advances in Neural Information Processing Systems}, 26, 2013.

\bibitem{ikeda2023nonlinear}
Masahiro Ikeda and Shun Uchida.
\newblock Nonlinear evolution equation associated with hypergraph laplacian.
\newblock {\em Mathematical Methods in the Applied Sciences}, 46(8):9463--9476,
  2023.

\bibitem{fukao2022heat}
Takeshi Fukao, Masahiro Ikeda, and Shun Uchida.
\newblock Heat equation on the hypergraph containing vertices with given data.
\newblock {\em arXiv preprint arXiv:2212.05446}, 2022.

\bibitem{shi2024hypergraph}
Kehan Shi and Martin Burger.
\newblock Hypergraph $ p $-laplacian regularization on point clouds for data
  interpolation.
\newblock {\em arXiv preprint arXiv:2405.01109}, 2024.

\bibitem{slepcev2019analysis}
Dejan Slepcev and Matthew Thorpe.
\newblock Analysis of p-laplacian regularization in semisupervised learning.
\newblock {\em SIAM Journal on Mathematical Analysis}, 51(3):2085--2120, 2019.

\bibitem{chambolle2011first}
Antonin Chambolle and Thomas Pock.
\newblock A first-order primal-dual algorithm for convex problems with
  applications to imaging.
\newblock {\em Journal of mathematical imaging and vision}, 40:120--145, 2011.

\bibitem{chambolle2018stochastic}
Antonin Chambolle, Matthias~J Ehrhardt, Peter Richt{\'a}rik, and Carola-Bibiane
  Schonlieb.
\newblock Stochastic primal-dual hybrid gradient algorithm with arbitrary
  sampling and imaging applications.
\newblock {\em SIAM Journal on Optimization}, 28(4):2783--2808, 2018.

\bibitem{zhang2017re}
Chenzi Zhang, Shuguang Hu, Zhihao~Gavin Tang, and TH~Hubert Chan.
\newblock Re-revisiting learning on hypergraphs: confidence interval and
  subgradient method.
\newblock In {\em International Conference on Machine Learning}, pages
  4026--4034. PMLR, 2017.

\bibitem{shor2012minimization}
Naum~Zuselevich Shor.
\newblock {\em Minimization methods for non-differentiable functions},
  volume~3.
\newblock Springer Science \& Business Media, 2012.

\bibitem{flores2022analysis}
Mauricio Flores, Jeff Calder, and Gilad Lerman.
\newblock Analysis and algorithms for $\ell$p-based semi-supervised learning on
  graphs.
\newblock {\em Applied and Computational Harmonic Analysis}, 60:77--122, 2022.

\bibitem{sen2008collective}
Prithviraj Sen, Galileo Namata, Mustafa Bilgic, Lise Getoor, Brian Galligher,
  and Tina Eliassi-Rad.
\newblock Collective classification in network data.
\newblock {\em AI magazine}, 29(3):93--93, 2008.

\end{thebibliography}

\end{document}